\documentclass[11pt]{article}

\usepackage{ifpdf}   %% pour commande diff{\'e}rente selon une compilation
\ifpdf  %-- si pdflatex
\usepackage{graphicx,pstricks,xypic,comment}
\else    %-- si latex
\usepackage{inputenc}
\usepackage[hypertex]{hyperref}      %% Pour les hyperliens entre le

\usepackage[active]{srcltx}          %% dvi et fichier tex
\usepackage{graphicx,pstricks,xypic,comment}
\DeclareGraphicsExtensions{.eps,.ps,.pst}        

\fi     %-- fin si

\newcommand{\executeiffilenewer}[3]{%
 \ifnum\pdfstrcmp{\pdffilemoddate{#1}}%
 {\pdffilemoddate{#2}}>0%
 {\immediate\write18{#3}}\fi%
}

\usepackage{color}
\usepackage{mathrsfs} % calligraphiques en mode math que l'on appelle
\usepackage{amssymb,amscd,amsmath,amsfonts,amsthm}

\theoremstyle{plain}

\newtheorem{theo}{Theorem}[section]
\newtheorem{lem}[theo]{Lemma}
\newtheorem*{lem*}{Lemma}

\newtheorem{prop}[theo]{Proposition}
\newtheorem{conj}[theo]{Conjecture}
\newtheorem*{conj*}{Conjecture}
\newtheorem*{theo*}{Theorem}

\theoremstyle{remark}

\newtheorem{rem}{Remark}

\theoremstyle{definition}

\newcommand{\ga}{\gamma}
\newcommand{\Si}{\Sigma}

\newcommand{\al}{\alpha}

\newcommand{\si}{\sigma}
\newcommand{\om}{\omega}
\newcommand{\Ga}{\Gamma}

\newcommand{\la}{\lambda} 
\newcommand{\Om}{\Omega}

\newcommand{\ep}{\varepsilon} 
\newcommand{\Hilb}{{\mathcal{H}}}

\newcommand{\C}{\mathbb C}
\newcommand{\R}{\mathbb R}
\newcommand{\N}{\mathbb N}
\newcommand{\Z}{\mathbb Z}
\newcommand{\T}{\mathbb T}

\newcommand{\su}{\operatorname{SU}_2}
\newcommand{\lu}{\operatorname{su}_2}

\newcommand{\mo}{{\mathcal{M}}}  % espace de module
\newcommand{\Ad}{\operatorname{Ad}} % representation adjointe
\newcommand{\tr}{\operatorname{Tr}}
\newcommand{\id}{\operatorname{id}}
\newcommand{\tor}{\operatorname{Tor}}
\newcommand{\Ci}{{\mathcal{C}}^{\infty}}  %%% Classe C infini 
\newcommand{\End}{\operatorname{End}} 

\newcommand{\CS}{\operatorname{CS}}  % Chern-Simons

\newcommand{\alt}{\operatorname{alt}} 

\newcommand{\ab}{\operatorname{ab}} 
\newcommand{\ir}{\operatorname{irr}}

\newcommand{\MS}{\operatorname{MS}}

\newcommand{\sL}{{\mathcal{L}}}

\newcommand{\sT}{{\mathcal{T}}} 

\newcommand{\dens}[1]{|\!\det #1|}
\newcommand{\bmo}{{\overline{\mathcal{M}}}}

\newcommand{\irr}{\textrm{irr}}
\newcommand{\so}{\operatorname{SO}_3}

\title{Knot state asymptotics II\\ Witten conjecture and irreducible representations}

\author{L. Charles\footnote{Institut de
    Math{\'e}matiques de Jussieu (UMR 7586), Universit{\'e} Pierre et
    Marie Curie -- Paris 6, Paris, F-75005 France.}\,\, and J. March{\'e}\footnote{Centre de math{\'e}matiques Laurent Schwartz (UMR 7640), Ecole Polytechnique -- 91128 Palaiseau, France} }

\date{}
\begin{document}

\maketitle

\begin{abstract}
This article pursues the study of the knot state asymptotics in the large level limit initiated in \cite{LJ1}. As a main result, we prove the Witten asymptotic expansion conjecture for the Dehn fillings of the figure eight knot. 

The state of a knot is defined in the realm of Chern-Simons topological quantum field theory as a holomorphic section on the $\su$-character manifold of the peripheral torus. 
%
% From Chern-Simons topological quantum field theory with gauge group $\su$ and level $k$, we associate to a knot exterior in $S^3$ an element in a vector space that we call the knot state. As the vector space is isomorphic to the geometric quantization of level $k$ of the $\su$-character variety of the peripheral torus, the knot state may be viewed as a holomorphic section over this character variety. We study its asymptotic properties in the large $k$ limit. 
%
In the previous paper, we conjectured that the knot state concentrates on the character variety of the knot with a given asymptotic behavior on the neighborhood of the abelian representations. In the present paper we study the neighborhood of irreducible representations. We conjecture that the knot state is Lagrangian with a phase and a symbol given respectively by the Chern-Simons and Reidemeister torsion invariants. We show that under some mild assumptions, these conjectures imply the Witten conjecture on the asymptotic expansion of WRT invariants of the Dehn fillings of the knot. 

Using microlocal techniques, we show that the figure eight knot state satisfies our conjecture  starting from $q$-differential relations verified by the colored Jones polynomials. The proof relies on a differential equation satisfied by the Reidemeister torsion along the branches of the character variety, a phenomenon which has not been observed previously as far as we know. 
\end{abstract}

\section{Introduction}
Using ideas coming from quantum field theory, E. Witten introduced in \cite{witten} a family of topological invariants of $3$-manifolds denoted by $Z_k(M)$. For any integer $k$, 
\begin{gather} \label{eq:partition_fonction}
Z_k(M)=\int \CS(A)^k \; \mathcal{D}A
\end{gather}
where $A$ is a 1-form with values in $\lu$, $\CS(A) \in U(1) $ stands for the Chern-Simons invariant of $A$ and $\mathcal{D}A$ denotes an hypothetical measure on the space of connections $\Om^1 ( M, \lu)$.
Doing the perturbative expansion of the partition function (\ref{eq:partition_fonction}), Witten derived the asymptotic behavior of $Z_k(M)$ in the large $k$ limit. The leading order term is a sum over the flat connections $A$ of $M$ which involves the Chern-Simons invariant $\CS (A)$ and the Reidemeister torsion $\T (A)$. Here the flat connections are considered up to gauge equivalence so that the relevant space is the moduli space $\mo (M)$ of representations $\rho:\pi_1(M)\to \su$ up to conjugation.

The invariants $Z_k(M)$  were later defined rigorously by Reshetikhin and Turaev in \cite{rt}. The asymptotic behavior predicted from the path integral is now referred to as the Witten asymptotic conjecture. Its interest (and difficulty) is that the gauge-theoretic quantities as the Chern-Simons invariant and Reidemeister torsion do not enter in any obvious way in the combinatorial definition of the $Z_k (M)$'s. References for this problem may be found in \cite{fg,je,apb}. 

In this article we will prove the conjecture for the Dehn fillings of the figure eight knot.   
\begin{theo}
Let $M$ be a manifold obtained by Dehn surgery on the figure eight knot with parameters $(p,q)$. Suppose that $p$ is not divisible by 4 and that for any irreducible representation $\rho:\pi_1(M)\to \su$ one has $H^1(M,\Ad_\rho)=0$. Then

\begin{xalignat*}{2} 
 Z_k (M) &  = \sum_{\rho \in \mo (M)}  e^{i \frac{m (\rho) \pi }{4}} k^{n(\rho)} a(\rho) \CS ( \rho )^k  + O(k^{-1}) 
\end{xalignat*}
where for any $\rho \in \mo (M)$,  $m( \rho)$ is an integer, $n( \rho)  = 0$, $-1/2$ or $-3/2$ according to whether $\rho$ is irreducible, abelian non-central or central. 
$$ a ( \rho) = \begin{cases} 2^{-1} \bigl( \T ( \rho) \bigr)^{1/2}  \text{ if $\rho$ is irreducible} \\ 
 2^{-1/2} \bigl( \T ( \rho) \bigr)^{1/2} \text{ if $\rho$ is abelian non-central} \\ 2^{1/2} \pi / p ^{3/2}    \text{ if $\rho$ is central.}\end{cases}$$ 
\end{theo}

We refer to Proposition \ref{prop:reg-slope} for a discussion on the hypothesis of this theorem. It holds true in many cases including all the slopes satisfying $|p/q| < 2 \sqrt 5$.
 One may object that the contribution of central representations is irrelevant because of the remainder $O(k^{-1})$. We decided to keep it in view of the more general Witten conjecture (see Theorem \ref{theo:witt-conj}).

The Witten asymptotic conjecture has been proved for a lot of Seifert manifolds. This study was initiated by Jeffrey in \cite{je}, followed by Rozansky \cite{ro} and Lawrence-Zagier \cite{lz}. Their results were generalized by Hikami \cite{hikami2} and Hansen-Takata \cite{ht}. All these works are based on explicit formulas for the Witten-Reshetikhin-Turaev invariants of Seifert manifolds. 
Using the geometric construction of the quantum representations of the mapping class groups, Andersen proved the conjecture for finite order mapping tori in \cite{andersen}  and the first author for mapping tori of diffeomorphisms satisfying a transversality assumption  \cite{l2}. Nevertheless the equivalence between the geometric and combinatorial constructions has not yet been proved. Progress towards proving the Witten conjecture for the Dehn filling of the figure eight knot was made by Andersen and Hanse \cite{ah} who showed that the conjecture would follow from certain integral representation of the asymptotics of the WRT invariants. However, to the best of our knowledge, this integral representation remains conjectural. We will not use this integral representation in our paper.

To our knowledge, our result is the first proof of the Witten asymptotic conjecture for some hyperbolic manifolds. The mapping tori considered in \cite{l2} include very likely hyperbolic manifolds but the invariant is defined as a trace which is only conjecturally equal to the WRT invariant. Furthermore our proof does not rely on an explicit formula or an integral representation of the WRT invariant. It is based on the $q$-difference equations satisfied by the colored Jones polynomials and uses microlocal techniques. We hope it can be generalized to every knot. 

Our strategy is to study the semi-classical properties of the family of knot states. More precisely, for any integer $k$ one can construct a topological quantum field theory (TQFT) as a functor from a cobordism category to the category of hermitian vector spaces, see \cite{bhmv} for instance. Given a knot $K$ in $S^3$, let $E_K$ be the complement of an open tubular neighborhood of $K$ and denote by $\Sigma$ its boundary. The TQFT at level $k$ associates to $\Sigma$ a finite dimensional hermitian space $V_k(\Sigma)$ and to $E_K$ a vector $Z_k(E_K)\in V_k(\Sigma)$. These vectors will be called the knot states. 

In the first part of this work, see \cite{LJ1}, we give an explicit isomorphism between $V_k(\Sigma)$ and the geometric quantization $\Hilb_k^{\alt}$ at level $k$ of the moduli space $\mo(\Sigma)$ of representations $\pi ( \Si) \rightarrow \su$ up to conjugation. This moduli space is a complex orbifold with four singular points corresponding to the central representations. It it the base of two holomorphic line (orbi-)bundles, the Chern-Simons bundle $L_{\CS}$ and a half-form bundle $\delta$ respectively. The quantum space $\Hilb_k^{\alt}$ is the space of holomorphic sections of $L^k _{\CS} \otimes \delta$.

 Let $r:\mo(E_K)\to\mo(\Sigma)$ be the natural restriction map induced by the inclusion $\Sigma\subset E_K$.  We will say that $ \rho \in \mo(E_K)$ is regular if $H^1(E_K,\Ad_\rho)$ has dimension 1.

\begin{conj} \label{conj:asymptotic_knot_state}
Let $K$ be a knot in $S^3$ and $Z_k(E_K)\in \Hilb_k^{\alt}$ be the family of knot states. For any $x \in \mo (\Si)$, we have the following trichotomy:
\begin{itemize}
\item[-] if $x \notin r(\mo(E_K))$ then  $Z_k(E_K)(x)=O(k^{-\infty})$.
\item[-] if $ r^{-1} ( x ) = \{ \rho\}$ and $\rho$ is regular irreducible, then
 $$Z_k(E_K)(x)\sim e^{im_x\frac\pi 4}\frac{k^{3/4}}{4\pi^{3/4}} \CS(\rho)^k \sqrt{\T(\rho)}$$
\item[-] if $ r^{-1} ( x ) = \{ \rho\}$ and $\rho$ is regular abelian non-central, then
$$Z_k(E_K)(x)\sim e^{im_x\frac\pi 4}\frac{k^{1/4}}{2^{3/2}\pi^{3/4}} \CS(\rho)^k \sqrt{\T(\rho)}$$
\end{itemize}
\end{conj}
We will actually state more precise results where the asymptotic equivalent at $x$ is replaced by an asymptotic expansion in a neighborhood of $x$, cf. Conjectures \ref{conj:microsupport}, \ref{conj:irreducible} and \ref{conj:abelian}. In the language of semi-classical analysis, these conjectures say that the knot state is a Lagrangian state supported by $r( \mo (E_K))$ with symbol and phase given respectively by a square root of the Reidemeister torsion and the Chern-Simons invariant. 
  
Let us comment the ingredients of Conjecture \ref{conj:asymptotic_knot_state}. 
By construction of the line bundle $L_{\CS}$, the Chern-Simons invariant $\CS(\rho)$ is a vector in the fiber of $L_{\CS}$ at $r( \rho)$. This is a particular case of the definition proposed in \cite{rsw} of the Chern-Simons invariant of a 3-dimensional manifold with non-empty boundary. 

The square root of the Reidemeister torsion $\T(\rho)$ is naturally a vector in the fiber of the half-form bundle $\delta$ at $r( \rho)$.  Indeed, at a regular representation $\rho\in\mo(E_K)$, the moduli space is smooth and its tangent space is isomorphic to $H^1(E_K,\Ad_\rho)$. The Reidemeister torsion $\T(\rho)$ is a linear form on $H^1(E_K,\Ad_\rho)$, well-defined up to sign.
Moreover, the restriction map $r:\mo(E_K)\to\mo(\Sigma)$ is a Lagrangian immersion on the regular part. We have two isomorphisms
$$\delta_\rho^{\otimes 2} \simeq \bigl( T^{1,0} _\rho \mo ( \Si) \bigr)^* \simeq  T_{\rho}^* \mo(E_K) \otimes \C $$ 
the first one being part of the definition of a half-form bundle, the second one being the restriction of the pull-back by the linear tangent map $T_{\rho} r$. There are exactly four elements $ \sqrt{\T(\rho)}\in\delta$ whose square is sent to $\pm \T(\rho)$ by these isomorphisms.

Conjectures \ref{conj:microsupport}, \ref{conj:irreducible} and \ref{conj:abelian} have many corollaries. They are compatible with the gluing operation and in particular, they allow to recover the usual Witten conjecture for Dehn fillings of the knot under some mild assumptions.  More precisely, let $M$ be the  3-manifold obtained by Dehn filling on $K$, that is $M=E_K\cup_\Sigma (-N)$ where $E_K$ is the knot exterior, $\Sigma$ the peripheral torus and $N$ is a solid torus with boundary $\Sigma$. The WRT invariant of $M$ is given by the scalar product:
\begin{gather} \label{eq:pairing}
Z_k(M)=\langle Z_k(E_K),Z_k(N)\rangle 
\end{gather} 
in $\Hilb_k^{\alt}$. We proved in Proposition 3.3 of \cite{LJ1} that $(Z_k(N))$ is a Lagrangian state supported by the image of the restriction map $r : \mo ( N) \rightarrow \mo ( \Si)$. Now one can estimate the scalar product of two Lagrangian state supported by transversal Lagrangian submanifold, the leading order term being given by a particular pairing in the half-form bundle, see \cite{l1}.
 In particular, if the knot $K$ satisfies Conjectures \ref{conj:microsupport}, \ref{conj:irreducible} and \ref{conj:abelian} and $r(\mo ( E_K) )$ intersects transversally $r( \mo ( N))$, then the Witten asymptotic conjecture holds, cf. Theorem \ref{theo:witt-conj} for a precise statement. 
Interestingly the pairing in the half-form bundle corresponds to the computation of the Reidemeister torsion by a Mayer-Vietoris argument.
The idea that Chern-Simons invariant and Reidemeister torsion are semi-classical data associated to 3-manifold with boundary and the fact that they behave as expected under gluing were already present in the work of Jeffrey and Weitsman but at a formal level, see \cite{jw}. 

We can deduce in the same way from Conjectures \ref{conj:microsupport}, \ref{conj:irreducible} and \ref{conj:abelian}  a generalized Witten conjecture where the solid torus  $N$ contains a banded knot colored by an integer $\ell$ going with $k$ to infinity, see Theorem \ref{conj-witten-gen}. In particular we get asymptotic expansions of the colored Jones polynomials in the large $k$ and $\ell$ limit. These asymptotics are parts of the so-called generalized volume conjecture, cf. \cite{murakami}.

Besides deriving consequences of  Conjecture \ref{conj:irreducible}, we show it for the figure eight knot. The case of the torus knots is proved in \cite{Ctoric}. Let us explain the main steps of the proof. We start from q-difference relations satisfied by the colored Jones polynomials $J_n ^K ( t) \in \C [ t^{\pm 1}]$ of the knot:
$$ P(M, L, t) J_n ^K = R ( t, t^n) .$$
Here $M$ and $L$ are operators acting on the sequence of colored Jones polynomial by $(Mf)_n=t^{2n}f_n$ and $(Lf)_n=f_{n+1}$, $P$ is a polynomial expression in $M$, $L$ and $t$, and $R$ is a rational function of two variables. 

Using that the coefficients of the knot state in a suitable basis are given by evaluation of the colored Jones polynomials, we deduce from each $q$-difference relation an equation of the form
\begin{gather} \label{eq:Toeplitz}
 P Z_k(E_K) = R 
\end{gather}
where now $P$ is a Toeplitz operator of $\Hilb_k^{\alt}$. Toeplitz operators have semi-classical properties similar to the ones of pseudodifferential operators depending on a small parameter. In particular, using the work of the first author \cite{l3} and \cite{l4} on the eigenstates of Toeplitz operator, we can deduce the asymptotic behaviour of $Z_k(E_K)$ from Equation  (\ref{eq:Toeplitz}) as follows.

Any Toeplitz operator has a principal and subprincipal symbol. In the case at hand,  these are functions defined on $\mo ( \Si)$ that we compute explicitly from the relation $P(M,L,t)$. For the $q$-differential relation that we will use, the principal symbol of $P$ vanishes on $r ( \mo^{\ir} ( E_K))$ whereas the right-hand side $R$ is a Lagrangian state supported by $ r ( \mo^{\ab} ( E_K))$. Here $ \mo^{\ir} ( E_K)$ and $ \mo^{\ab} ( E_K)$ are respectively the subsets of $\mo ( E_K)$ consisting of irreducible and abelian representations. 

Let $x \in \mo (\Si)$. Assume first that the principal symbol does not vanish at $x$. Then we can invert $P$ on a neighborhood of $x$. If in addition $x$ does not belong to the microsupport of $R$, we deduce from Equation (\ref{eq:Toeplitz}) that the knot state is a $O(k^{-\infty})$ at $x$. If on the other hand $x$ belongs to $r( \mo ^{\ab} (E_K))$, we deduce  that $Z_k (E_K)$ is on a neighborhood of $x$ a Lagrangian state supported by $r( \mo ^{\ab} (E_K))$, the symbol of this Lagrangian state being the quotient of the symbol of $R$ by the principal symbol of $P$. For the torus and figure eight knots, we can compute this symbol in terms of the Alexander polynomial. In this way we prove in \cite{LJ1} the first and last assertions of Conjecture \ref{conj:asymptotic_knot_state}. 

Assume now that the principal symbol of $P$ vanishes at $x$ but its differential does not.  Assume in addition that $x \notin r( \mo^{\ab}( E_K)$ so that we can neglect $R$. Then on a neighborhood of $x$, Equation (\ref{eq:Toeplitz}) determines the knot state up to a multiplicative constant and a $O(k^{-\infty})$ term. It implies that $Z_k(E_K)$ is a Lagrangian state supported by the zero level set of the principal symbol of $P$. The phase of this Lagrangian state being a flat section of the Chern-Simons bundle, it is equal to the Chern-Simons invariant up to a constant. The symbol of this Lagrangian state satisfies a transport equation governed by the principal and subprincipal symbol of $P$. We show that the Reidemeister torsion of the figure eight and torus knots satisfies this transport equation. The last step is to compute the multiplicative constants, one for each component of $\mo ^{\ir} (E_K)$. Using symmetries of the figure eight knot state, we show that only one initial value is needed. We obtain it from exceptional surgeries. As it is well-known, the $\pm 1$-surgeries on the figure eight knot yield the Brieskorn sphere $\Sigma(2,3,7)$. The Witten asymptotic conjecture was proved by Hikami in that case in \cite{hikami2}.
Then the constant is determined by Equation  (\ref{eq:pairing}). This proves the second assertion of Conjecture \ref{conj:asymptotic_knot_state}. 

This strategy should work for other knots. By order of difficulty, it seems reasonable to apply it to twist knots and  two-bridge knots. The general case seems out of reach as far as we do not have a good understanding of q-difference equations satisfied by the colored Jones polynomial.

As already mentioned, we prove that the Reidemeister torsion of the figure eight and torus knots satisfies a transport equation deduced from a q-difference relation of the colored Jones polynomial. We believe that this is a new phenomenon which deserves further investigation.

In Section 2, we construct the knot state, summing up a construction which details can be found in \cite{LJ1}. We describe the symmetry properties of the knot state. We then review in Section 3 the topological invariants entering into the picture, that is representation spaces, Chern-Simons invariant and Reidemeister torsion. We recall the gluing and symmetry properties, and give some explicit formulas. In Section 4, we formulate our main conjecture on the semi-classical properties of the knot state in the neighborhood of irreducible representations, extending the main conjectures of \cite{LJ1}. We then show that this conjecture implies the Witten conjecture for the Dehn fillings of the knot and some generalization. 
Finally, Section 5 is devoted to the proof of the main conjecture for the figure eight knot. We go over the tools of microlocal analysis which are needed and discuss the transport equation of the Reidemeister torsion of the figure eight knot. We end the proof with considerations of symmetries and exceptional surgeries. The paper ends with a discussion on transversality assumptions needed for the proof of the Witten conjecture for the Dehn fillings on the figure eight knot.

{\bf Acknowledgements:}

We would like to thank Fr{\'e}d{\'e}ric Faure, Gregor Masbaum and the ANR team "Quantum Geometry and Topology" for their interest and valuable discussions. The second author was supported by the French ANR project ANR-08-JCJC-0114-01.

\section{Knot state} \label{sec:knot-state}

\subsection{Geometric quantization}\label{sec:geom-quant}

Let $(E, \om)$ be a symplectic vector space with a lattice $R$ of volume $4 \pi$. Our aim is to quantize the quotient of $E$ by the group  $R \rtimes \Z_2$, where $R$ acts by translation and the generator of $\Z_2$ acts by $-\id_E$. Since this quotient has four singular points, we will merely quantize $E$ with its symmetry group $R \rtimes \Z_2$. 

Introduce a complex linear structure $j$ on $E$ compatible with the symplectic structure. 
Let $( \delta, \varphi)$ be a half-form line, that is $\delta$ is a complex line and $\varphi$ an isomorphism from $\delta^{\otimes 2}$ to the canonical line $K_j = \{ \al \in E^* \otimes \C / \al ( j \cdot) = i \al \}$.  Let us denote also by $\delta$ the trivial holomorphic line bundle over $E$ with fiber $\delta$. We lift the action of $R \rtimes \Z_2$ to $\delta$ by $$(x,\ep ) ( y,v) = ( x+  \ep y , \ep v).$$ 
Let $\al \in \Om^1 ( E, \C)$ be given by $\al_x ( y) = \frac{1}{2} \om (x,y)$. Let $L$ be the trivial hermitian line bundle over $E$ endowed with the connection $d + \frac{1}{i} \al$. The bundle $L$ has a unique holomorphic structure compatible with $j$ and this connection.  Let the Heisenberg group be $E \times U(1)$ with the product:
\begin{gather} \label{eq:prod_Heisenberg} 
 (x,u).(y,v) = \Bigl( x+y, uv \exp \Bigl( \frac{i}{2} \om (x,y)\Bigr)\Bigr).
\end{gather}
The same formula with $(y,v) \in L$ defines an action of the Heisenberg group on $L$ which preserves the connection and the holomorphic structure.
Because the volume of $R$ is $4 \pi$, $R \times \{1\}$ is a subgroup of the Heisenberg group. Lifting trivially the action of $\Z_2$ on  $E$, we get an action of $R \rtimes \Z_2$ on $L$.

For any positive integer $k$, we have a representation of the group $R \rtimes \Z_2$ on the space of holomorphic sections of $L^k \otimes \delta$. The space quantizing the quotient $E / R \rtimes \Z_2$ is the $R \rtimes \Z_2$-invariant subspace. We denote it by $\Hilb_k^{\alt}$ and view it as a subspace of the space $\Hilb_k$ consisting of $R$-invariant holomorphic sections. $\Hilb_k$ has a natural scalar product defined by: 
\begin{gather} \label{eq:scalar_product}
 \langle \Psi_1 , \Psi_2 \rangle  = \int_{ D} \langle \Psi_1 (x), \Psi_2 (x) \rangle_{\delta} \; |\om | (x), \qquad \Psi_1, \Psi_2 \in \Hilb_k 
\end{gather}
where $D$ is any fundamental domain of $R$.

\subsection{Jones polynomial} 

Let $K$ be a knot in $S^3$. Let $J_\ell^K \in \Z [ t^{\pm1}]$ be the Jones polynomial of $K$ colored with the $\ell$-dimensional irreducible representation of $\mathfrak{sl}_2$ and normalized so that it is equal to the quantum integer 
$$[\ell] = \frac{t^{2 \ell} - t^{-2 \ell}}{t^2 - t^{-2}}$$ 
for the unknot, cf. \cite{LJ1} Section 4.1. 

Let $\Si$ be the peripheral torus of $K$. The 2-dimensional vector space $E= H_1(\Si, \R)$ has a natural symplectic product given by $4 \pi$ the intersection product. The lattice $R =  H_1(\Si, \Z)$ has volume $4 \pi$. Introduce a complex structure and a half-form line of $E$ as in the previous section and define the two quantum spaces $\Hilb_k$ and $\Hilb_k^{\alt}$. 

Let $\mu$ and $\la \in R$ be a meridian and a longitude respectively. The elements $(\mu/2k, 1)$ and $(-\la/2k,1)$ of the Heisenberg group act on the sections of $L^k \otimes \delta$ and preserve the subspace $\Hilb_k$. We denote by $M$ and $L$ their restriction to $\Hilb_k$. Let $\Om_\la \in \delta$ be such that $\Om_\la ^2 (\la) =1$. Then one proves (see Theorem 2.2, \cite{LJ1}) that $\Hilb_k$ has a unique orthonormal basis $(\Psi_\ell)_{\ell \in \Z/ 2k \Z}$ such that 
\begin{gather} \label{eq:def_base_hilbk} 
 M \Psi_\ell = e ^{i \ell  \frac{\pi}{k}} \Psi_\ell , \qquad L \Psi_\ell = \Psi_{\ell -1}
\end{gather}
and 
$$ \Psi_0 ( 0 ) = \Theta_k ( 0, \tau) \Bigl( \frac{k}{2 \pi} \Bigr)^{k/2} \Om_\mu .$$ 
Here $\tau$ is the parameter of the complex structure of $E$ and $\Theta_k$ a theta series. Its precise value is not important for our purpose, only the fact that $\Theta_k (0, \tau) = 1 + O(e ^{-k/C})$ for some constant $C>0$.  

The state of the knot $K$ is the following vector of $\Hilb_k$:
\begin{gather} \label{eq:etat_noeud}
 Z_k(E_K) = \frac{\sin(\pi/k)}{\sqrt{k}}\sum_{\ell \in \Z / 2k
  \Z} J_\ell ^K(-e^{i\pi/2k}) \Psi_\ell 
\end{gather}
We denote it by $Z_k(E_K)$ because it corresponds in topological quantum field theory to the vector associated to the exterior $E_K$ of the knot. 

\subsection{Symmetries of the knot state}

First the knot state is an alternate section, that is it belongs to $\Hilb_K ^{\alt}$. This comes from the fact that $J_{-\ell}^K=-J_{\ell}^K$ and that the generator of $\Z_2 \subset R \rtimes \Z_2$ acts on $\Hilb_k$ by sending $\Psi_\ell $ to $-\Psi_{-\ell}$. 

 There is a less obvious symmetry corresponding to the translation with vector $\la/2$.  Denote by $T^*_{\la/2}$ the pull-back operator by the action of the element $( \la/2,1)$ of the Heisenberg group. We claim that
$$ T^*_{\la/2} Z_k (E_K) = - Z_k(E_K)$$
Indeed it follows from the characterization (\ref{eq:def_base_hilbk}) of the basis $(\Psi_{\ell})$ of $\Hilb_k$ that $T^*_{\la/2} \Psi_\ell = \Psi_{\ell + k}$. So this additional symmetry is a consequence of the following proposition. 

\begin{prop} \label{sec:symm-knot-state}
For any $k\in \Z_{>0}$, write $t_k=-e^{i\pi/2k}$. Then for any $\ell\in \Z$ one has 
$$J^K_{\ell+k}(t_k)=-J^K_{\ell}(t_k).$$ 
\end{prop}
\begin{proof}
We refer to \cite{bhmv} for the construction of the objects used in this proof. For any $\ell$ in $\N$, consider the family of Jones-Wenzl idempotents $f_\ell\in \sT_\ell$ where $\sT_\ell$ is the Temperley-Lieb algebra with $\ell$ points. These elements are defined over $\mathbb{Q}(t)$ and we can specialize $t=t_k$ in $f_\ell$ provided that $\ell\le k-1$. The $\ell$-th colored Jones polynomial evaluated at $t_k$ is obtained by cabling the knot $K$ by $\ell-1$ strands, inserting the idempotent $(-1)^{\ell-1}f_{\ell-1}$ and computing the Kauffman bracket at $t=t_k$. 

The closure of the idempotent $(-1)^{\ell-1}f_{\ell-1}$ in the skein module of the solid torus gives the polynomial $T_\ell\in \Z[x]$ where $x^\ell$ stands for $\ell$ parallel copies of the core and $T_\ell$ is defined by 
$$T_0=0, \quad T_1=1,\quad T_{\ell+1}+xT_{\ell}+T_{\ell-1}=0$$
One proves by induction that for any $\ell\in \Z$, 
\begin{gather} \label{eq:rec_rel} 
T_{k-\ell}=T_{k-1}T_\ell-T_kT_{\ell-1}.
\end{gather} 
Due to the "killing property" of $f_{k-1}$,  $\langle T_kT_{\ell-1}\rangle$ vanishes when it is evaluated at $t_k$. Indeed,  the idempotent property implies that anything outside the idempotent can be reduced up to a multiplicative factor to the standard closure $\langle f_{k-1}\rangle$ of the idempotent. Then the formula $$\langle f_{k-1}\rangle=(-1)^{k-1}\frac{t_k^k-t_k^{-k}}{t^2-t^{-2}}=0$$ shows the result. 

Let us prove now that $J^K_{k-1}=1$. A simple computation of fusion rule shows that changing a crossing of a bunch of $k-2$ strands with $f_{k-2}$ inserted with a single strand amounts in multiplying by $-1$ (evaluated at $t_k$), see Figure \ref{fig:transparent}. 

\begin{figure}[htbp]
\centering
 \def\svgwidth{6cm}
 \executeiffilenewer{transparent.svg}{transparent.pdf}%
 {inkscape -z -D --file=transparent.svg %
 --export-pdf=transparent.pdf --export-latex}%
 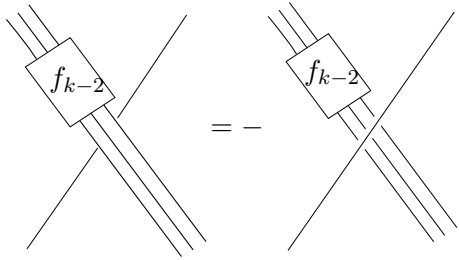%

  \caption{Transparent idempotents}
  \label{fig:transparent} 
\end{figure}

As the knot $K$ can be untied with an even number of crossing changes, we get $$J^K_{k-1}(t_k)=\frac{t_k^{2k-2}-t_k^{-2k+2}}{t_k^2-t_k^{-2}}=\frac{\sin(\pi-\pi/k)}{\sin(\pi/k)}=1.$$
By a similar reasoning, one shows the "transparency property" of the idempotent $f_{k-2}$ and deduces that $\langle T_{k-1}T_\ell\rangle$ is equal to $\langle T_\ell\rangle$ when evaluated at $t_k$.  Indeed, assume the knot $K$ is cabled by two idempotents $f_{k-2}$ and $f_{l-1}$. If we change a crossing between the copy cabled by $f_{k-2}$ and the copy cabled by $f_{l-1}$, the result is multiplied by $(-1)^{l-1}$. By changing an even number of crossings, we can separate the two parallel copies and get $$\langle T_{k-1}T_\ell\rangle=\langle T_{k-1}\rangle\langle T_\ell\rangle=\langle T_\ell\rangle .$$ 
So by Equation (\ref{eq:rec_rel}), $\langle T_{k-\ell} \rangle = - \langle T_\ell\rangle $, which ends the proof.
\end{proof}

\section{Topological invariants}\label{sec:top-inv}
In this section, we will review some well-known facts about representation spaces, Chern-Simons invariants and Reidemeister torsion.
Unless these theories make sense for any 3-manifolds with boundary, we will restrict ourselves to either closed 3-manifold or knot exteriors.

\subsection{Representation spaces} \label{sec:rep-spaces}

For any compact manifold $N$, the moduli space $$\mo ( N)=\textrm{Hom}(\pi_1(N),\su)/\su$$ is a real algebraic variety. Let $\rho:\pi_1(N)\to \su$ be a representation. We denote by $\Ad_\rho$ the vector space $\lu$ where $\gamma\in\pi_1(N)$ acts on $\xi\in\lu$ by $\gamma.\xi=\Ad_{\rho(\gamma)}\xi$. Let $H^*_\rho(N)$ be the cohomology of $N$ with twisted coefficients in $\Ad_{\rho}$.
It is well-known that for any $\rho \in \mo (N)$, the Zariski tangent space at $\rho$ is identified with the first cohomology group $H_\rho^1 (N)$ of $N$ with coefficient in $\Ad \rho$, cf. as instance \cite{hk} p.40.

Consider on $\lu$ the Euclidean pairing $\langle A,B\rangle=\tr(A^*B)$. When $N$ is a $n$-dimensional oriented manifold with possibly empty boundary, the Poincar{\'e} duality in the twisted case gives a non-degenerate pairing 
\begin{equation}\label{eq:poincare-dualite}
H^k_\rho(N)\times H^{n-k}_\rho(N,\partial N)\to\R.
\end{equation}
\subsubsection{The torus case} 
Let $\Si$ be an oriented 2-dimensional torus. The smooth part $\mo ^s ( \Si)$ of $\mo ( \Si)$ consists on the non-central representations. 
Let $\pi$ be the map  from $E= H_1(\Si, \R)$ onto $\mo ( \Si)$ defined by
$$\pi(x)(\gamma)=\exp((\gamma\cdot x) D)\quad \forall \gamma\in H_1(\Sigma,\Z)$$
where $\cdot$ stands for the intersection product and $D$ is the diagonal matrix with entries $2i\pi,-2i\pi$. Let $R = H_1( \Si, \Z)$ and consider the action of $R \rtimes \Z_2$ on $E$ as in Section \ref{sec:geom-quant}. The projection $\pi$ factors through a bijection between the quotient of $E$ by $R \rtimes \Z_2$ and  $\mo(\Si)$.

Let us describe the tangent map to $\pi$ at $x \in H_1( \Si , \R)$. First $H_1( \Si, \R)$ is isomorphic with $H^1( \Si, \R)$ by the  Poincar{\'e} duality mapping $y\in H_1(\Si,\R)$ to the cocycle $z\to z\cdot y$. Then, the inclusion $\R\to \Ad_{\pi(x)}$ mapping $t$ to $tD$ induces a map $H^1(\Sigma,\R)\to H^1_{\pi(x)}(\Sigma)$. 
The tangent map is the composition 
$$T_x\pi:H_1(\Si,\R) \simeq H^1( \Si, \R) \to H^1_{\pi(x)}(\Si).$$
When $\pi (x) $ is not central, that is $x \!\notin\! \frac{1}{2} R$, $T_x \pi$  is easily shown to be an isomorphism. Furthermore, the action of $R \rtimes \Z_2$ on $E^s = E \setminus \frac{1}{2} R$ being proper and free, the quotient $E^s / R \rtimes \Z_2$ is smooth and  $\pi$ induces a diffeomorphism from $E^s / R\rtimes \Z_2$ to  $\mo ^{s} ( \Si)$. 

Finally, if $\rho \in \mo(\Si)$ is not central, the pairing on $H^1_{\rho}(\Si)\simeq H_1(\Sigma,\R)$ given by Equation \eqref{eq:poincare-dualite} is a symplectic form whose integral over a fundamental domain for $H_1(\Sigma,\Z)$ is $\tr(D^*D)=8\pi^2$. So it is equal to $2\pi\omega$, where $\om$ is the symplectic form considered in Section \ref{sec:geom-quant}. 

\subsubsection{The knot exterior case}

Let $K$ be a knot in $S^3$ and denote as previously its exterior by $E_K$ and its peripheral torus by $\Si$.  Since $\Si$ is contained in $E_K$, we have a restriction  map $r:\mo(E_K)\to\mo(\Si)$. The set $\mo(E_K)$ is the disjoint union of $\mo ^{\ab}(E_K)$ and $\mo ^{\irr}(E_K)$, which consist respectively in the abelian and irreducible representations.

The moduli space $\mo ^{\ab}(E_K)$ is homeomorphic to a closed interval. The restriction of $r$ to $\mo ^{\ab}(E_K)$ is injective and its image is $\pi ( [0,1] \la)$, where $\la \in H_1(\Si, \Z)$ is a longitude.  
An abelian representation $\rho \in \mo ( E_K)$ is said to be regular if the eigenvalues of $\rho ( \mu)^2 $, with $\mu$ a meridian of $K$,  are not root of the Alexander polynomial of $K$. 

We say that an irreducible representation $\rho \in \mo(E_K)$ is regular if its restriction to the boundary $r( \rho)$ is not central and if the vector space $H^1_\rho(E_K)$ is one-dimensional. If it is the case, the morphism $H^1_\rho ( E_K) \rightarrow H^1_{r ( \rho)}(\Si)$ induced by the inclusion $\Si \subset E_K$ is injective. Conversely, if $\rho$ satisfies this last condition and $r( \rho)$ is not central, then $\rho$ is regular, cf. \cite{hk} p. 42. 

The set $\mo^s ( E_K) $ consisting of irreducible regular representations is a smooth open one-dimensional submanifold of $\mo(E_K)$. Its tangent space at $\rho$ is $H^1_\rho (E_K)$. Furthermore the map $r$ restricts into an immersion from $\mo^ s(E_K)$ to $ \mo ( \Si)$, the tangent map at $\rho \in \mo ^s (E_K)$ being the morphism $H^1_\rho ( E_K) \rightarrow H^1_{r ( \rho)}(\Si)$ induced by the inclusion $\Si \subset E_K$.

\subsubsection{Symmetries on representation spaces}
Denote by $\rho_{\pm 1}\in \mo(E_K)$ the central representations mapping the meridian to $\pm 1$. As the product of any representation with a central one is again a representation, one gets an action of $\Z_2$ on $\mo(E_K)$. Its quotient $\bmo(E_K)$ can be identified with the set of representations of $E_K$ in $\so$. 
Indeed two representations in $\mo(E_K)$ are the same as representations of $\so$ if and only if they differ by a central representation. Moreover, any representation in $\so$ lifts to $\su$, as the obstruction to the existence of such a lifting is in $H^2(E_K,\Z_2)=0$. 

Consider the quotient $\bmo(\Si)=E/\Gamma'$ where $\Gamma'=R'\rtimes \Z_2$ and $R' = \mu \Z \oplus \frac{1}{2} \lambda \Z$. This quotient is in intermediate position between the representation spaces of $\Si$ in $\su$ and $\so$. It has the nice property that the restriction map descends to a map $\overline{r}$ in the following diagram:

$$\xymatrix{
\mo(E_K)\ar[r]^r\ar[d]& \mo(\Si)\ar[d]\\
\bmo(E_K)\ar[r]^{\overline{r}}& \bmo(\Si)}$$
The topological invariants introduced in the next subsection will descend to these quotients. 
\subsection{Chern-Simons invariants}\label{sec:chern-simons-invar}

\subsubsection{Generalities}

Let $M$ be a manifold with a possibly non-empty boundary and of dimension not greater than $3$.  Then any representation $\rho\in\mo(M)$ is the holonomy of a flat connection $\alpha\in \Omega^1 (M,\lu)$. This connection form is unique up to gauge transformation, that is any other one has the form 
$$\alpha^g=g^{-1}\alpha g+g^{-1}\mathrm{d}g$$ 
for some $g:M\to\su$.
In other terms, we may identify the moduli space $\mo ( M)$ with the quotient of the space $\Omega^1_{\flat}(M,\lu)$ of flat connections by the action of $\Ci(M, \su)$.  

Assume that $M$ is compact and 3-dimensional. We define the Chern-Simons functional by the following formula:
\begin{equation}\label{def-CS}
\CS(\al)=\exp \Bigl( \frac{1}{12i\pi}\int_{M}\tr(\alpha\wedge\alpha\wedge\alpha) \Bigr), \qquad \al \in \Om^1_{\flat} ( M, \lu)
\end{equation}
If $M$ is closed, a well-known computation shows that $\CS( \al^g) = \CS ( \al)$, cf \cite{freed-cs}. So we can define the Chern-Simon invariant of $\rho \in \mo (M)$ by  $\CS( \rho) = \CS ( \al)$.  

Assume now that the boundary $\Si$ of $M$ is not empty. Then Equation (\ref{def-CS}) does not give a gauge invariant quantity, but $$ \CS (\al^g) \CS( \al)^{-1} =  c(a,h)$$ with $a$ and $h$ the restrictions of  $\al$ and $g$ respectively  to $\Si$ and 
$$ c (a,h)  = \exp \Bigl( iW(h) +\frac{1}{4i\pi} \int_{\Si} \tr(a\wedge \mathrm{d} h h^{-1}) \Bigr)$$
where $W(h)=\frac 1\pi \int_{M}\tr((g^{-1}\mathrm{d}g)^{\wedge 3})$ is the Wess-Zumino-Witten functional of $h$ which is independent on $g$ modulo $2\pi$.

The vector space  $\Om^1 ( \Si , \lu)$ has a symplectic form 
$$ \Om (\al , \beta) = - \frac{1}{2 \pi} \int_{\Si} \tr (\al \wedge \beta ) $$
The quotient of the subset of all flat connections by the gauge group action can be viewed as a symplectic reduction. The trivial line bundle $\Om^1 ( \Si , \lu) \times \C$ equipped with the connection 
$ \Theta_A (\alpha) = \frac{1}{2} \Om (A, \al)  $
is a prequantum bundle. 
The gauge group $\Ci ( \Si, \su)$ acts on $\Om^1 ( \Si , \lu) \times \C$ by $$h.(\al,z) = ( \al^h, c(\al,h) z ).$$ 
Restricting to the flat connections and dividing by this action, we get a line bundle $L_{\CS} \rightarrow \mo ( \Si)$, called the Chern-Simons line bundle. 

Let $r$ be the restriction map from  $\mo ( M)$ to $\mo ( \Si)$ induced by the inclusion of $\Si$ into $M$. By construction, for any representation $\rho \in \mo (M)$, the family $\CS(\al)$ where $\al$ runs over the flat connections with holonomy $\rho$,  defines a vector $\CS(\rho)$ in the fiber at $r( \rho)$ of $L_{\CS}$. Furthermore, the connection descends to the Chern-Simons bundle and the section $\rho \rightarrow \CS ( \rho)$ of $r^* L_{\CS}$ is flat. Here to avoid the singularities, we should restrict everything to the subspace of irreducible connections when $\Si$ has a genus $\geqslant 2$, and to the non-central ones when $\Si$ has genus 1.  We refer the reader to \cite{freed-cs} for more details. 

Consider two compact oriented manifolds $M_1$ and $M_2$ with boundary identified with $\Si$. Let $M$ be the gluing of $M_1$ with $-M_2$ along $\Si$. Then one checks easily that for any $\rho \in \mo (\Si)$, 
\begin{gather} \label{eq:gluing_Chern_Simons}
 \CS( \rho) = \langle \CS ( \rho_1), \CS ( \rho_2) \rangle
\end{gather}
where $\rho_1$ and $\rho_2$ are the restriction of $\rho$ to $M_1$ and $M_2$ respectively. The bracket in the right hand side is the scalar product in the fiber of $L_{\CS}$ at the restriction of $\rho$ to $\Si$. 

\subsubsection{Comparison between bundles} 

Assume now that $\Si$ is a torus so that $\mo ( \Si) \simeq E/ R \rtimes \Z_2$. Recall that we introduced in Section \ref{sec:geom-quant} a prequantum bundle $L \rightarrow E$. 

\begin{lem} 
The quotient of $L$ by $\Gamma=R \rtimes \Z_2$ is isomorphic to the Chern-Simons bundle $L_{\CS}$. The restrictions  of these two bundles over $\mo ^s ( \Si)$ have isomorphic connections. 
\end{lem}

\begin{proof} 
We can represent any representation $\rho \in \mo ( \Si) $ by an element in $\Om^1_\flat( \Si, \su)$ of the form $b D $ with $D =\operatorname{diag} ( 2i\pi, -2i \pi)$ and  $b$ a real valued 1-form on $\Si$. It is unique up to a gauge transformation of the form $h_1 h_2$ where 
$$ h_1 = \begin{pmatrix} 0 &1\\-1&  0 \end{pmatrix} \text { or } \id, \qquad h_2 =  \begin{pmatrix}e^{2i\pi H}&0\\0&e^{-2i\pi H}\end{pmatrix} $$
where $H$ is a map from $\Si$ to $\R / \Z$. Identify $E$ with $H^1( \Si, \R)$ so that the de Rham class of $b$ defines an element of $E$ satisfying $\pi ( [b]) = \rho$. Observe furthermore that 
$$ (bD)^{h_1} = \pm b D, \qquad (bD)^{h_2} = ( b + d H) D
$$
which proves again that $\mo ( \Si)$ is the quotient of $E$ by $R \rtimes \Z_2$. 
One shows that $c( bD, h_1)= 1$ and 
\begin{xalignat*}{2} 
c( bD, h_2) = & \exp \Bigl( 2 i \pi \int_\Si b \wedge dH \Bigr) \\
 = & \exp \Bigl( \frac{i}{2} \om ( [b], [dH]) \Bigr)   
\end{xalignat*}
Comparing with equation (\ref{eq:prod_Heisenberg}), the result follows. 
\end{proof}

Assume now that $M$ is a knot exterior $E_K$, so that $\Si$ is the peripheral torus of the knot. Then the map sending $\rho \in \mo (E_K)$ to $\CS ( \rho)$ is a section of the bundle $r^* L_{CS} \rightarrow \mo ( E_K)$. Its restrictions to $\mo^{s} ( E_K)$ and $\mo^{\ab} (E_K)$ are flat. 

\subsubsection{Symmetry}

The group $\Ga = R \rtimes \Z_2$ is an index 2 normal subgroup of $\Gamma'=R'\rtimes \Z_2$. Here $R' = \mu \Z \oplus \frac{1}{2} \lambda \Z$.  We extend the action of $\Ga$ on the bundle $L$ to $\Ga'$ in such a way that 
$$ (\la /2  , 1  ) .( x, v)  =  (  x +  \la/2 , e^{\frac{i}{2} \om (\la/2, x )} v) , \qquad  (x, v) \in L
$$  
Denote by $\tau$ and $\tau_L$ respectively the actions of the generator of $\Ga'/ \Ga$ on $\mo ( \Si ) \simeq E/ \Ga$ and $L_{\CS} \simeq L / \Ga$. We denote $\si$ the action of $\rho_{-1}$ on $\mo (E_K)$.
\begin{lem} \label{sec:symmetry-CS}
For any representation $\rho\in\mo(E_K)$ one has
$$\CS(\si (\rho) )=\tau_L (\CS(\rho))$$
\end{lem}
\begin{proof}
Let $F\subset E_K$ be a Seifert surface of $K$ such that $F\cap \Si=\partial F=\lambda$. Let $\alpha\in\Omega^1_\flat(E_K,\lu)$ be a flat connection representing $\rho$. Thicken $F$ such that there is an embedding $j:F\times [0,1]\to E_K$. Then, up to gauge transformation, one can suppose that 
\begin{itemize}
\item[-] 
$\alpha=bD$ on the boundary where $b\in\Omega^1(\Si,\R)$ and $D$ is the diagonal matrix with entries $2i\pi$ and $-2i\pi$.
\item[-]
$j^*\alpha=p^*\beta$ where $p:F\times[0,1]\to F$ is the projection on the first factor and $\beta$ is a flat connection in $\Omega^1_\flat(F,\lu)$.
\end{itemize}

Let  $\phi$ be a smooth function on $[0,1]$ which is equal to 0 on a neighborhood of 0 and 1 on a neighborhood of 1. Then one can construct a connection $\alpha'$ which represents $\tau(\rho)$ in the following way: $\alpha'$ coincide with $\alpha$ outside $F\times[0,1]$ and satisfies on $F \times [0,1]$
$$j^*\alpha'=e^{-D\phi/2}(p^*\beta) e^{D\phi/2}+\frac{1}{2}D\mathrm{d}\phi.$$ 
Using the formula defining the Chern-Simons invariant, we get 
$$\CS(\alpha')=\CS(\alpha)\exp(\theta) \quad \text{ with } \theta=\frac{3}{12i\pi}\int_{F\times[0,1]} \tr \Bigr( \frac{D}{2}\mathrm{d}\phi\wedge \beta\wedge \beta \Bigr).$$ 
Integrating over $[0,1]$ and using the flatness of $\beta$, we get 
$$\theta=\frac{i}{8\pi}\int_F \tr(D\mathrm{d}\beta)=\frac{i}{8\pi}\int_{\partial F}b\tr(D^2)=-i\pi\int_{\lambda}b=\frac{i}{2}\omega(\lambda/2,[b]).$$
One recognizes the action of $\tau_L$ on the pair $([b],\CS(\alpha))$ and the lemma is proved.
\end{proof}

\subsection{Reidemeister torsion} \label{sec:reidemeister-torsion}

Let $M$ be an oriented manifold with possibly empty boundary. We will be interested in the cases when $M$ is a closed 3-manifold, a 2-torus or a knot exterior although Reidemeister torsion makes sense in the general case. 
\subsubsection{Construction}
Given a real vector space $V$, we denote by $\dens{V}$ the vector space of densities on $V$. 
Consider on $\lu$ the Euclidean pairing $\langle A,B\rangle=\tr(A^*B)$ and denote by $\nu\in \dens{\lu}$ the Euclidean density. 

Suppose that $M$ is homeomorphic to a finite connected CW-complex and let $\rho:\pi_1(M)\to \su$ be a representation. We denote by $\Ad_\rho$ the vector space $\lu$ where $\gamma\in\pi_1(M)$ acts on $\xi\in\lu$ by $\gamma.\xi=\Ad_{\rho(\gamma)}\xi$.

For a lighter notation, we will denote by $C^*_\rho(M)$ the complex $C^*(M,\Ad_{\rho})$. It is a finite complex isomorphic to a direct sum of copies of $\lu$ (one for each cell of $M$). Define $$\dens{C^*_\rho(M)}=\bigotimes_i \dens{C^i_\rho(M)}^{(-1)^{i+1}}$$
This line has a generator obtained by taking a convenient tensor product of copies of $\nu$ and its inverse, this generator is well-defined up to sign. 
Using the well-known isomorphism between $\dens{C^*_\rho(M)}$ and $\dens{H^*_\rho(M)}$ we obtain a generator of the latter space that we denote by $\tor(M,\rho)$. This generator does not depend on the way $M$ is presented as a cellular complex.

The torsion $\T(M,\rho)$ is constructed from $\textrm{Tor}_\rho(M)$ by a procedure which depends on the pair $(M,\rho)$. We review here the cases which occur in this article.
As $M$ is connected, one may view $H^0_\rho(M)$ as a subspace of $\lu$. Hence it inherits from $\lu$ an Euclidean structure and a corresponding volume element denoted by $v(N,\rho)\in \dens{H^0_\rho(M)}^{-1}$. 
Using the pairing on $\lu$ and Poincar{\'e} duality, we see that for a manifold $M$ of dimension $n$ and for $k\le n$, there is a natural isomorphism $H^k_\rho(M)\simeq H^{n-k}_\rho(M,\partial M)^*$. Hence, if $M$ is closed, we can associate to $v(M,\rho)$ a density $\nu(M,\rho)$ in $\dens{H^n_\rho(M)}$. 

\begin{enumerate}
\item
If $M$ is closed 3-manifold, we will suppose that $H^1_\rho(N)=0$ which means that $\rho$ is (infinitesimally) isolated in $\mo(M)$. By Poincar{\'e} duality, one has also $H^2_\rho(N)=0$. 
Hence we have an element $\tor(M,\rho)$ in $\dens{H^*_\rho(M)}=\dens{H^0_\rho(M)}^{-1}\otimes\dens{H^3_\rho(M)}$. We define $\T(M,\rho)\in \R$ by the equation $$\tor(M,\rho)=v(M,\rho)\nu(M,\rho)\T(M,\rho).$$ Observe that if $\rho$ is irreducible, the normalizations $v$ and $\nu$ are useless.
\item
If $M$ is a torus, we have $\tor(M,\rho)\in \dens{H^0_\rho(M)}^{-1}\otimes\dens{H^1_\rho(M)}\otimes\dens{H^2_\rho(M)}^{-1}$. Define in that case $\T(M,\rho)$ by the formula $$\tor(M,\rho)=v(M,\rho)\T(M,\rho)\nu(M,\rho)^{-1}.$$
It is well-known that the density $\T(M,\rho)$ coincides with the symplectic density associated to the Poincar{\'e} pairing on $H^1_\rho(M)$ that we described in Subsection \ref{sec:rep-spaces}, see \cite{witten-2d} p.187.
\item
If $M$ is a knot exterior and $\rho$ is regular abelian representation with non-central restriction to the boundary then 
we have $H^2_\rho(M)=H^3_\rho(M)=0$. So $\tor(M,\rho)\in \dens{H^0_\rho(M)}^{-1}\otimes\dens{H^1_\rho(M)}$. We define $\T(M,\rho)$ by 
$$ \tor(M,\rho)=v(M,\rho)\T(M,\rho)
$$
It is a density on the one-dimensional space $H^1_\rho(M)$. It can be computed in terms of the Alexander polynomial, cf. Section \ref{sec:examples}. 
\item 
If $M$ is a knot exterior and $\rho$ is an regular irreducible representation then we have $H^0_\rho(M)=H^3_\rho(M)=0$. So $\tor(M,\rho)\in \dens{H^1_\rho(M)}\otimes\dens{H^2_\rho(M)}^{-1}$. Because of the regularity assumption, $H^1_\rho(M)$ is 1-dimensional,  the map $$r^*:H^1_\rho(M)\to H^1_{r(\rho)} (\partial M)$$ is injective and  the map $$r^*:H^2_\rho(M)\to H^2_{r(\rho)} (\partial M)$$ is an isomorphism, see \cite{hk}, p.42.
We define $\T(M,\rho)$ by
$$\tor(M,\rho)=\T(M,\rho)(r^*)^{-1}(\nu(\partial M,\rho)^{-1}).$$
It is a density on the one-dimensional space $H^1_\rho(M)$.
\end{enumerate}

To summarize, we obtained for any closed 3-manifold $M$ equipped with a representation $\rho$ such that $H^1_\rho (M)$ is trivial, a numerical invariant $\T(M,\rho)$.  Furthermore for any regular representation $\rho$ of some knot exterior $M$, we defined a density on $H^1_\rho(M)$.
\begin{rem}
The Reidemeister torsion at a point $\rho\in \mo(E_K)$ is defined through the adjoint representation $\Ad_{\rho}$. So it depends only on the projection of $\rho$ in $\bmo(E_K)$.
\end{rem}
\subsubsection{Examples} \label{sec:examples}

Let $a$ and $b$ be two coprime integers. The torsion of the lens spaces $L(a,b)$ was computed by Franz in \cite{franz}. The fundamental group of $L(a,b)$ is $\Z/a\Z$: for $n\in\Z/a\Z$, we let $\rho_n:\Z/a\Z\to \su$ be the representation mapping 
the generator to the matrix $\exp(nD/a)$. Denote by $b^*$ an inverse of $b\mod a$. Then $$\T(L(a,b),\rho_n)=\frac{16}{a} \Bigl| \sin \Bigl( \frac{2\pi n}{a} \Bigr) \sin \Bigl( \frac{2\pi b^* n}{a} \Bigr) \Bigr|.$$

Let $E_K$ be the complement of a knot $K$ in $S^3$. Denote by $\rho_q$ the abelian representation of $E_K$ mapping the meridian to the matrix $\exp(qD)$. Let $\Delta_K$ be the normalized Alexander polynomial. Then it is shown in Theorem 4 of \cite{milnor1} that:
\begin{gather} \label{eq:torsion_abelian}
\T(E_K,\rho_q)=\frac{4\sin^2(2\pi q)}{|\Delta_K(\exp(4i\pi q))|^2}2^{3/2}\pi  |r^* {\rm d} q|.
\end{gather}
Here $r$ is the restriction map from $\mo (E_K)$ to $\mo ( \Si)$, $p$ and $q$ are the coordinates on $\mo (\Si)$ such that one has $r(\rho)=\pi(p\mu+q\la)$. We use the same notation in the two following examples. 

Let $a$ and $b$ be two coprime integers and $E_{a,b}$ be the complement in $S^3$ of the torus knot with parameters $a,b$. The fundamental group of $E_{a,b}$ is 
$$  \pi_1( E_{a,b}) \simeq \langle x,y | \; x^a=y^b\rangle.$$ Any irreducible representation $\rho$ of $E_{a,b}$ is regular, its Reidemeister torsion being given by 
$$\T(E_{a,b},\rho)=\frac{16}{a^2b^2}\sin^2 \Bigl( \frac{\pi\ell}{a} \Bigr )\sin^2 \Bigl( \frac{\pi m}{b} \Bigr ) 2^{3/2}\pi |r^* \mathrm{d}p|.$$ 
where $m$ and $\ell$ are two integers such that $\tr(\rho(x))=2\cos(\pi \ell/a)$ and $\tr(\rho(y))=2\cos(\pi m/b)$. 

Let $E_8$ be the complement of the figure eight knot. Then any irreducible representation $\rho$ of $E_8$ is regular (see Proposition 4.19 in \cite{porti}) and
\begin{gather}\label{eq:torsion_8}
\T(E_8,\rho)=\frac{2^{3/2}\pi |r^* \mathrm{d}p|}{1-4\cos(4\pi q)}
\end{gather}
where $r( \rho)=\pi(p\mu+q\la)$.

\subsubsection{Gluing formula}
Let $E_1,E_2$ be two knot complements with boundaries identified to a torus $\Si$. 
Let $\rho_j$ be elements of $\mo(E_j)$ which are either regular abelian or regular irreducible and whose restrictions on $\mo(\Si)$ coincide with a representation $\rho$ which is not central. 
Corresponding to the decomposition $M=E_1\cup (-E_2)$, there is a short exact sequence of complexes induced by restriction maps: $$0\to C^*_{\tilde{\rho}}(M)\to C^*_{\rho_1}(E_1)\oplus C^*_{\rho_2}(E_2)\to C^*_{\rho}(\Si)\to 0.$$
This sequence provides the following isomorphism: 
$$\dens{H^*_\rho(M)}\otimes \dens{H^*_{\tilde{\rho}}(\Si)}\simeq \dens{H^*_{\rho_1}(E_1)}\otimes\dens{H^*_{\rho_2}(E_2)}.$$
The well-known formula linking the four torsions is 
\begin{equation}\label{mult-torsion}
\tor(E_1,\rho_1)\tor(E_2,\rho_2)=\tor(M,\tilde{\rho})\tor(\Si,\rho)\tor(H),
\end{equation}
 where $H$ is the long Mayer-Vietoris exact sequence, see \cite{milnor}, Theorem 3.2. A simple argument shows that the correction term $\tor(H)$ is trivial, see \cite{freed-torsion}, Lemma 1.18.

These considerations imply the following proposition:
\begin{prop} \label{prop:gluing-formula-Rei}
Let $M$ be a 3-manifold obtained by gluing two knot exteriors $E_1$ and $E_2$ along a torus $\Si$. Let $\tilde{\rho}$ be a representation in $\mo (M)$ which restricts to $\rho_1,\rho_2$ and $\rho$ on $E_1,E_2$ and $\Si$ respectively. Suppose that $\rho_1$ and $\rho_2$ are regular, one of the two is abelian, $\rho$ is not central, and $H^1_{\tilde{\rho}}(M)=0$. Then from the Mayer-Vietoris sequence, one has an isomorphism $$H^1_\rho(\Si)\simeq H^1_{\rho_1}(E_1)\oplus H^1_{\rho_2}(E_2)$$ induced by the two restriction maps $r_1$ and $r_2$. Denoting by $\pi_1,\pi_2$ the corresponding projections, one has:
$$\T(M,\rho) \T(\Si,\rho)=\pi_1^*\T(E_1,\rho_1)\wedge \pi_2^*\T(E_2,\rho_2)$$
\end{prop}

\begin{proof}
The proof follows directly from Equation \eqref{mult-torsion} by checking that the normalization terms cancel. The less obvious cancellation comes from the case when all representations are reducible. We check that the normalizations $\nu(M,\tilde{\rho})$ and $\nu(\Si,\rho)^{-1}$ cancel via the boundary map $\partial:H^2_\rho(\Si)\to H^3_{\tilde{\rho}}(M)$.
\end{proof}

\section{Witten conjectures and generalization}

\subsection{Knot state asymptotics}

Consider a knot $K$ with exterior $E _K$ and peripheral torus $\Si$. The state of $K$ is the vector $Z_k(E_K)$  given by (\ref{eq:etat_noeud}). It belongs to the vector space $\Hilb_k^{\alt}$ quantizing $E/ R \rtimes \Z_2$, where $E= H_1( \Si, \R)$ and $R = H_1(\Si, \Z)$. So it may be viewed as a $R \rtimes \Z_2$-invariant section over $E$, or as a $\Z_2$-invariant section over $E/R$ or even as a section of an orbifold bundle over the moduli space $\mo ( \Si) \simeq E/ R \rtimes \Z_2$. In the sequel we use these three different representations according to our needs.

\subsubsection{Microsupport} 
Recall first that $Z_k (E_K)$ is admissible, that is there exist $N$ and $C$ such that  
$$ \| Z_k (  E _K ) \| \leqslant C k^{N} $$ 
for any $k$ (see \cite{LJ1}, Section 5.5). Define the microsupport of any admissible family $(\Psi_k \in \Hilb_k , \; k \in \Z_{>0})$ as the subset $\MS (\Psi_k)$ of $E$ such that for any $x\in E$, $x$ does not belong to $\MS (\Psi_k)$ if and only if there exists a neighborhood $U$ of $x$ and a sequence of positive number $(C_N)$ such that for any $N$ and for any $k$
$$ | \Psi_k (y) | \leqslant C_N k^{-N}, \qquad \forall y \in U .$$
The microsupport is $R$-invariant. If the $\Psi_k$'s are alternate, it is $R \rtimes \Z_2$-invariant and so it can be viewed as a subset of $\mo (\Si) = E/ R\rtimes \Z_2$. In \cite{LJ1}, we conjectured the following statement: 

\begin{conj} \label{conj:microsupport} 
The microsupport of $(Z_k (E _K))$ is contained in $r ( \mo (E _K))$. 
\end{conj}
The conjecture has been proved for the eight knot \cite{LJ1} and for the torus knots \cite{Ctoric}. To complete this, we make two conjectures on the asymptotic behavior of the knot state on a neighborhood of $r ( \mo (E _K))$. 

\subsubsection{Irreducible representation} 
Here it is more convenient to consider the knot state as a section over the moduli space $\mo ( \Si)$. Recall first that the smooth part $\mo^s ( \Si )$ is diffeomorphic to $E^s / R \rtimes \Z_2$. The symplectic and the complex structure of $E$  descend to $\mo^s ( \Si)$. The quotient of $L \rightarrow E^s$ is the the restriction of the Chern-Simons bundle $L_{\CS}$. The quotient  of $\delta \rightarrow E^s$ is a line bundle over $\mo^s ( \Si)$ that we still denote by $\delta$. 
Each section $\Psi \in \Hilb_k^{\alt}$ defines a holomorphic section  of $L_{\CS}^k \otimes \delta \rightarrow \mo^s ( \Si)$. We could actually extend $\delta$ and $L_{\CS}$ to orbibundles over $\mo ( \Si)$ in such a way that $\Hilb_k^{\alt}$ gets identified with the space of holomorphic sections of $ L^k_{\CS} \otimes \delta \rightarrow \mo ( \Si)$, but it is not necessary for our purposes.  

The morphism $ \varphi: \delta^2 \rightarrow K_j$ does not descend to the quotient $\mo^s ( \Si)$, because $R \rtimes \Z_2$ acts on the canonical line $K_j$ by $(x,n) .z = (-1)^n z$. But the square $\varphi^2$ descends to an isomorphism between $\delta^4 \rightarrow \mo^s ( \Si)$ and the square of the canonical bundle of $\mo^ s ( \Si)$. We still use $\varphi$ for notational purpose even if it is only defined up to a sign. 

\begin{conj}\label{conj:irreducible}
For any open set $U$ of $\mo ^{s}(\Si)$ such that $V = r^{-1} ( U)$ is connected, contained in $\mo^s(E_K)$  and $r$ restricts to an embedding from $V$ into $U$, we have on $U$
$$Z_k (E_K)   =   e ^{i \frac{m\pi}{4}} \frac{k^{3/4}}{4\pi^{3/4}} F^k f(\cdot ,k ) + O(k^{-\infty})$$
where $m$ is an integer
\begin{itemize} 
\item[-] $F$ is a section of $L \rightarrow U$ such that  $F(r(\rho)) = \CS (\rho)$ for any $\rho \in V$ and which satisfies the Cauchy-Riemann equation up to a term vanishing to infinite order along $r( \mo ^{s} )$. 
\item[-] $f(\cdot , k)$ is a sequence of $\Ci (U,\delta)$ admitting an
  asymptotic expansion of the form $f_0 + k^{-1} f_1 +  \cdots$ for the
  $\Ci$ topology with coefficients $f_i \in \Ci (U, \delta)$. Furthermore $$ \bigl( r^* \varphi( f_0^2) \bigr)( \rho)    =  \pm \T(\rho)$$ for any $\rho \in V$.
\end{itemize}
\end{conj}

Observe that for any regular irreducible representation $\rho \in \mo ( E_K)$ such that $r^{-1}( r ( \rho) ) = \{ \rho \}$,   $r( \rho)$ has a neighborhood $U$ satisfying the assumption of the conjecture. Furthermore the conjecture implies that 
\begin{gather} \label{eq:CS_Reid_irr}
 Z_k (E_K) (r ( \rho))  \sim e^{ i \frac{m \pi}{4}}\frac{k^{3/4}}{4\pi^{3/4}} \CS(\rho)^k
\tau ^{\frac{1}{2}} 
\end{gather}
where $\tau \in \delta^2_{r( \rho)}  $ is such that $r^* \varphi ( \tau) =\pm \T ( \rho)$. Hence, the Chern-Simons invariant and the torsion at $\rho$ are determined by the asymptotic behavior of $Z_k (E_K)$. 

\subsubsection{Abelian representations}

Our second conjecture describes the knot state on a neighborhood of any regular abelian representation. Denote by $\Delta_K$ the Alexander polynomial of $K$. We say that a point $x \in \la \R$ is regular if a neighborhood of $\pi(x)$ does not meet $r( \mo ^{\ir})$ and $\Delta_K ( e^{4i \pi q }) \neq 0$ if $ x = q \la$. In the following statement, we consider $Z_k (E_K)$ as an $R \rtimes \Z_2$-invariant section of $L^k \otimes \delta \rightarrow E$. 

\begin{conj}\label{conj:abelian}
Any regular point of $ \R \la$ has an open neighborhood $V$ in $E$ such that $V \cap \R \la$ consists of
regular points and 
$$ Z_k( E_K) (x)   =   e^{im \frac{\pi}{4}} \Bigl( \frac{k}{2 \pi} \Bigr)^{1/4}  t_\la^k(x) \otimes f(x ,k ) \Om_{\la} + O(k^{-\infty}) , \qquad \forall x \in V $$
where $m$ is an integer,  
\begin{itemize} 
\item[-] $t_\la$ is the holomorphic section of $L\rightarrow E$ restricting to $1$ on $\R \la$.
\item[-] $f(\cdot , k)$ is a sequence of $\Ci (V)$ admitting an
  asymptotic expansion of the form $f_0 + k^{-1} f_1 +  \cdots$ for the
  $\Ci$ topology with coefficients $f_i \in \Ci (V)$, the first one satisfying  $$f_0(q\la) =  \frac{1}{\sqrt{2}}  \frac{ \si - \si^{-1}  }{\Delta_K (\si^2)}  \quad \text{ with } \quad \si = e ^{ 2i \pi q }  $$ for any $q \la \in V$.
\item[-] $\Om_{\la}\in \delta$ such that $\Om_\la ^2 ( \la) = 1$.   
\end{itemize}
\end{conj}

Let $\rho \in \mo (E_K )  $ be an abelian representation so that $r(\rho) =\pi ( x)$ for some $x \in \R \la$. Then $\CS ( \rho) = t_\la ( x)$. Assume furthermore that $r( \rho)$ is regular  and non-central, then the Reidemeister torsion $\T( \rho)$ is given in terms of the Alexander polynomial by (\ref{eq:torsion_abelian}). Conjecture \ref{conj:abelian} implies that 
\begin{gather} \label{eq:CS_Reid_ab}
 Z_k (E_K) (r ( \rho))  \sim   e^{ i \frac{m \pi}{4}}  \frac{k^{1/4}}{2^{3/2}\pi^{3/4}}\CS(\rho )^k \tau^{1/2} ,
\end{gather}
where $\tau \in \delta^2_{r( \rho)}  $ is such that $r^* \varphi ( \tau) =\pm \T ( \rho)$.

\subsection{The Witten Conjecture for Dehn fillings}

In this part we show that under some mild assumptions, the previous conjectures imply the Witten conjecture about the asymptotic expansion of the WRT invariants of the Dehn fillings of a knot exterior. Consider a knot $K$ in $S^3$ and two relatively prime integers $p,q$.  The Dehn filling of $K$ with parameters $(p,q)$ is 
$$ M = E_K \cup_{\phi} (-N) $$ 
where $N$ is the solid torus $D^2 \times S^1$ and $\phi$ an oriented diffeomorphism $\partial N \rightarrow \Si$ sending the homology class of $\partial D^2$ into $p \mu + q \la$. Here $\mu$ and $\la$ are the homology classes of a meridian and longitude of $K$. 

Consider the segment
$$I_{p/q}  = \pi( (p\mu + q \la) \R ) \subset \mo( \Si) .$$ 
One has to assume that $I_{p/q}$ intersects transversally $r(\mo( E_K))$ at points where we can describe the state $ Z_k (E_K)$ with the previous conjectures. The precise hypothesis are the following.
\begin{enumerate} 
\item[H1-] $p \neq 0$ and $I_{p/q} \cap I_{0}$ consists of regular abelian points.
\item[H2-] $Z= I_{p/q} \cap r( \mo ^{\ir} ( E_K))$ is finite and for any of $x \in Z$, $r^{-1}(x) $ consists of a single regular irreducible representation $\rho$ which satisfies for any generator $\xi\in H^1_\rho(E_K)$:
$$\langle r^*\xi,p\mu+q\lambda\rangle \ne 0$$
\end{enumerate}
The fundamental group of $M$ is the quotient of the fundamental group of $E_K$ by the subgroup generated by the homotopy class of $p \mu + q \la$. So there is a one to one correspondence between $\mo (M)$ and the set of $\rho \in \mo(E_K)$ such that $r( \rho) \in I_{p/q}$. In particular, under the previous hypothesis, $\mo(M)$ is finite.  

The hypothesis (H1) and (H2) can be completely transcribed in terms of the moduli space $\mo(M)$, as it is explained in the following lemma.
\begin{lem}
Let $K$ be a knot in $S^3$ and $p,q$ be two relatively prime integers. Let $M$ be the Dehn filling of the exterior of $K$ with parameters $(p,q)$ and $L$ be the core of the solid torus glued to the exterior of $K$ in order to obtain $M$. Then, the hypothesis (H1) and (H2) are satisfied if and only if 
\begin{enumerate}
\item[H1'-] The restriction map $\mo(M)\to \mo(L)$ is injective.
\item[H2'-] For any non central representation $\rho\in \mo(M)$, one has $H^1_\rho(M)=0$. 
\end{enumerate}
\end{lem}
\begin{proof}
Write as above $M=E_K\cup_{\phi} (-N)$. The knot $L$ is the core of $N$ hence the moduli space $\mo(L)$ is identified with $\mo(N)$ which is itself identified to $I_{p/q}$ by the restriction mapping $\mo(N)\to\mo(\Sigma)$. 
The representation space $\mo(M)$ is in bijection with the set of representations $\rho\in\mo(E_K)$ such that $r(\rho)\in I_{p/q}$. The assumption that $r^{-1}(r(\rho))=\{\rho\}$ for all such representations is equivalent to the hypothesis H1'. 

Let $\rho$ be such a representation in $\mo(E_K)$. Abusing notation, we also denote by $\rho$ the corresponding representations in $\mo(M),\mo(N)$ and $\mo(\Sigma)$. The Mayer-Vietoris sequence of the gluing $M=E_K\cup_{\phi} (-N)$ gives: 

$$\xymatrix{&H^1_{\rho}(M)\ar[r]& H^1_{\rho}(E_K)\oplus H^1_{\rho}(N)\ar[r]& H^1_{\rho}(\Sigma)\\
0\ar[r]&H^0_{\rho}(M)\ar[r]& H^0_{\rho}(E_K)\oplus H^0_{\rho}(N)\ar[r]& H^0_{\rho}(\Sigma)\ar[ull]^{\partial}}$$
If the hypothesis $H1$ and $H2$ are satisfied, then the exactness of the bottom line shows that $\partial=0$. The geometric interpretation of $ H^1_{\rho}(E_K),H^1_{\rho}(N)$ and $H^1_{\rho}(\Sigma)$ as the tangent spaces of $\mo(E_K),I_{p/q}$ and $\mo(\Sigma)$ respectively shows that $H^1_\rho(M)$ is identified with the intersection of the tangent spaces of $r(\mo(E_K))$ and $I_{p/q}$ at $\rho\in\mo(\Sigma)$. The transversality assumption implies that this latter space vanishes. 
Reciprocally, if we suppose that $H^1_{\rho}(M)=0$ for all non central $\rho$, we get that $H^1_{\rho}(E_K)$ has dimension 1, showing that $\rho$ represents a regular element of $\mo(E_K)$ and the geometric interpretation holds again, implying the transversality assumption.
\end{proof}
\begin{theo} \label{theo:witt-conj}
Let $K$ be a knot satisfying Conjectures \ref{conj:microsupport}, \ref{conj:irreducible} and \ref{conj:abelian}. Let $p$ and $q$ be two relatively prime integers and $M$ be the Dehn filling of $K$ with parameters $p,q$. Assume that $(H1)$ and $(H2)$ are verified. Then 
\begin{xalignat*}{2} 
 Z_k (M) &  = \sum_{\rho \in \mo (M)}  e^{i \frac{m (\rho) \pi }{4}} k^{n(\rho)} \la_k(\rho) \CS ( \rho )^k  + O(k^{-\infty}) 
\end{xalignat*}
where for any $\rho \in \mo (M)$,  $m( \rho)$ is an integer, $n( \rho)  = 0$, $-1/2$ or $-3/2$ according to whether $\rho$ is irreducible, abelian non-central or central.  Furthermore $( \la_k ( \rho))$ is a sequence of complex numbers admitting an asymptotic expansion of the form 
$$ \la_k ( \rho) =   a_0 ( \rho) + a_1 ( \rho ) k^{-1} + a_2 ( \rho) k^{-2} + \ldots $$
with coefficient  $a_\ell ( \rho )  \in \C$, the leading one being given by  
$$ a_0 ( \rho) = \begin{cases} 2^{-1} \bigl( \T ( \rho) \bigr)^{1/2}  \text{ if $\rho$ is irreducible} \\ 
 2^{-1/2} \bigl( \T ( \rho) \bigr)^{1/2} \text{ if $\rho$ is abelian non-central} \\ 2^{1/2} \pi / p ^{3/2}    \text{ if $\rho$ is central.}
\end{cases} $$
\end{theo}

Because of the anomalies, the WRT invariant is well-defined only up to a power of $\tau_k = e^{ \frac{3 i \pi }{4} - \frac{3 i \pi }{2k}}$. Since $ \tau_k = e^{\frac{3 i \pi }{4}} + O(k^{-1})$, this factor appears in the leading order term in the power $m( \rho)$. 

For the abelian representation, the Chern-Simons invariant is easily computed exactly as for the lens spaces (cf. \cite{LJ1}, proof of Theorem 6.2)
 $$ \CS ( \rho_\ell) = e^{ 2 i \pi \ell^2 q/p} , \qquad \ell = 0 , 1, \ldots, p-1 $$
where $\rho_\ell \in \mo ^{\ab} (E_K)$ is such that $r ( \rho_\ell) = \pi ( \ell \frac{q}{p} \la)$. Similarly if $\ell \neq 0$ and $p/2$, then
$$ \T ( \rho_\ell) = \frac{2}{ \sqrt p }\frac{ \sin ( 2 \pi \ell /p ) \sin ( 2 \pi q \ell /p ) }{\Delta_K ( e^{4 i \pi q \ell /p})}$$
where $\Delta_K$ is the Alexander polynomial of the knot. Furthermore one has for some integer $m$ that: $m ( \rho_\ell) = m -2 $ if $\rho_\ell$ is central and $m ( \rho_\ell) = m$ otherwise.

\begin{proof}
We use in this article a family of topological quantum field theories denoted by $(V_k,Z_k)$ which associate to any surface $\Sigma$ a complex hermitian vector space $V_k(\Sigma)$ and to any 3-manifold $N$ bounding $\Sigma$ a vector $Z_k(M)\in V_k(\Sigma)$. The key property that we shall use in the proof is the gluing formula stating that for any two 3-manifolds bounding $\Sigma$ denoted by $N_1,N_2$, one has $Z_k(N_1\cup(-N_2))=\langle Z_k(N_1),Z_k(N_2)\rangle$.
In \cite{LJ1}, Theorem 2.4, we identified the hermitian space $V_k(\Sigma)$ with the space $\Hilb_k^{\alt}$ in a natural way, that is compatible with the (projective) actions of SL$_2(\Z)$ on both spaces. The knot states $Z_k(E_K)$ and $Z_k(N)$ in $\Hilb_k^{\alt}$ are compatible with these identifications, hence we have $$Z_k (M) = \langle Z_k ( E_K), Z_k (N) \rangle $$

To estimate this scalar product, we consider $Z_k (E_K)$ and $Z_k(N)$ as sections over the torus $T = E/R$. 

By Theorem 3.3 in \cite{LJ1}, the microsupport of $Z_k (N)$ is contained in $I_{p/q}$. By Conjecture \ref{conj:microsupport}, the microsupport of $Z_k (E_K)$ is contained in $r( \mo ( E_K))$. Hence for any compact neighborhood $C$ of $r ( \mo ( E_K)) \cap I_{p/q}$ in $\mo ( \Si)$, we have 
\begin{gather}   \label{eq:integrale} 
  \langle Z_k ( E_K), Z_k (N) \rangle =  \int_{\tilde{C}} \bigl( Z_k (E_K)   , Z_k (N) \bigr)_{L^k \otimes \delta}  (x) \;  \mu_T ( x)  +O(k^{-\infty}) 
\end{gather}
where $( \cdot, \cdot  )_{L^k \otimes \delta}$ denote the pointwise hermitian product of $L^k \otimes \delta \rightarrow T$, $\mu_T$ is the Liouville measure of $T$ and $\tilde C$ is the preimage of $C$ by the projection $T \rightarrow \mo ( \Si)$. 
 
Let $\ga = p \mu + q \la$. We proved in Theorem 3.3 of \cite{LJ1} that $Z_k(N)$ is a Lagrangian state supported by the circle $ C_{p/q} = \{ t\ga, \; t \in \R \} \subset T$. More precisely, for any real $t$ and $x = [t \ga] \in C_{p/q}$,       
\begin{gather}\label{eq:tore_solide_1}
 Z_k ( N) ( x) = e^{i \frac{\pi}{4} m} \Bigl( \frac{2 k}{ \pi }\Bigr)^{1/4}   \sin ( 2 \pi t) \Om_{ \ga} \otimes s_{\ga } ( x ) + O(k^{-3/4}) 
\end{gather}  
where $m$ is an integer, $\Om_{\ga} \in \delta$ is such that $ \Om^2 _{\ga } (\ga) =1$ and $s_{\ga}$ is the flat section of $L \rightarrow C_{p/q}$ lifting to the constant section of $L \rightarrow \ga \R$ equal to $1$. Now consider the representation $\rho \in \mo (N)$ whose restriction to $\Si$ is $\pi (x)$. Assume that $\rho$ is not central, then equation (\ref{eq:tore_solide_1}) may be rewritten as
\begin{gather} \label{eq:tore_solide_2}
 Z_k ( N) ( x) = e^{i \frac{\pi}{4} m} \frac{k^{1/4}}{2^{3/2}\pi^{3/4}}    \tau \otimes \CS ^k( \rho ) + O(k^{-3/4})
\end{gather}
where $\tau \in \delta_{r_N( \rho)} $ is such that $r^*_N \tau^2 = \T ( \rho)$. Here $r_N$ is the restriction map $ \mo ( N) \rightarrow \mo ( \Si)$. 

By our assumptions, on a neighborhood of each point of $r ( \mo ( E_K)) \cap I_{p/q}$, $Z_k(N)$ and $Z_k(E_K)$ are Lagrangian states supported by two transversal curves. So we can estimate the integral (\ref{eq:integrale}) by a pairing formula, cf. Proposition 6.1 of \cite{LJ1}.  This leads to 
$$ Z_k (M) = \sum _{\rho \in \mo (M)} \CS ^k (\rho) k^{n_\rho} (a_{\rho, 0} + a_{\rho,1} k^{-1} + \ldots ) + O(k^{-\infty}) $$  
Here we used the fact that the restriction map $ \mo ( M) \rightarrow \mo ( \Si)$ gives a bijection between $\mo (M)$ and $ r ( \mo ( E_K)) \cap I_{p/q}$, the abelian representations being sent in $
r ( \mo ^{\ab}( E_K)) \cap I_{p/q}$ and the irreducible ones in $r ( \mo ^{\irr}( E_K)) \cap I_{p/q}$. The Chern-Simons invariant $\CS( \rho)$ comes from the gluing formula (\ref{eq:gluing_Chern_Simons}). The power $n_\rho$ is equal to  $0$ or $-1/2$ according to $\rho$ is abelian or irreducible. 

The leading coefficients $a_{\rho,0}$ are computed as follows.  
If $\rho$ is abelian and non-central, then by (\ref{eq:CS_Reid_ab}) and (\ref{eq:tore_solide_2}), 
\begin{xalignat*}{2}
a_{\rho, 0 } = & \pm 2 \frac{2 \pi }{(2^{3/2} \pi^{3/4})^2}  \Biggl( \frac{\T ( r_{E_K}(\rho)) \T(r_N(\rho))}{\om }\Biggr)^{1/2} \\
= & \pm 2^{ - 1 /2} \bigl( \T (\rho) \bigr)^{1/2}  
\end{xalignat*}
The factor $2$ comes from the fact that the projection $E^s/R \rightarrow \mo ^s ( \Si)$ is two to one. The maps $r_{E_{K}}$, $r_{N}$ and $r_{\Si}$ are the restriction map from $\mo (M)$ to $\mo (E_K)$, $\mo ( N)$ and $\mo ( \Si)$ respectively. We used that $\om = (2\pi)^{-1}\T(r_{\Sigma} ( \rho)) $ and Proposition \ref{prop:gluing-formula-Rei} to obtain the last equality. 

In the case when $\rho$ is an irreducible representation, we deduce similarly from (\ref{eq:CS_Reid_irr}) and (\ref{eq:tore_solide_2}) that 
\begin{xalignat*}{2} 
a_{\rho, 0 } = &  \pm 2 \frac{2\pi}{(4\pi^{3/4})(2^{3/2}\pi^{3/4})}  \Biggl( \frac{\T ( r_{E_K}(\rho)) \T(r_N(\rho))}{\om }\Biggr)^{1/2} \\
= & \pm 2^{ - 1} \bigl( \T (\rho) \bigr)^{1/2}  .
\end{xalignat*}
Finally if $\rho$ is central, $a_{0,\rho}=0 $ because the sinus in (\ref{eq:tore_solide_1}) vanish. One can compute the second coefficient in the asymptotic expansion exactly as we did for the lens spaces, cf. proof of Theorem 6.2 of \cite{LJ1}. 

\end{proof}

\subsection{Some generalization} 

Consider a knot $K$ in $S^3$ . Denote by $E_K$ its exterior and by $\Si$ its peripheral torus. Let $N$ be the solid torus $D^2 \times S^1$ and $L$ be the banded link $[0, 1/2] \times S^1 \subset N$. Let $\phi$ be a diffeomorphism from $S^1 \times S^1$ to $\Si$ preserving orientations. Consider the Witten-Reshetikhin-Turaev invariant 
$$ Z_{k, \ell} := Z_k ( E_K \cup_{\phi}(-N), L, \ell, 0 ).$$ 
where the color $\ell$ is any integer satisfying $0<\ell<k$. 

Let $( \mu', \la')$ be the basis of $H_1 ( \Si)$ given by $\mu'= \phi ( S^1 \times \{ 1 \} )$ and $\la' = \phi ( \{ 1 \} \times S^1)$. For any real $\dot q \in (0,\frac 1 2)$, consider the circle
$$ C_{\dot q} = \pi ( \dot q \la' + \R \mu ' ) \subset \mo ( \Si) .$$
The following assumptions are similar to the ones of the previous part
 \begin{enumerate} 
\item[I1-] $\mu' \neq \la$ and $X_{\dot q} ^{\ab} = C_{\dot q}  \cap I_{0}$ consists of regular abelian points.
\item[I2-] $X_{\dot q } ^{\ir} = C_{\dot q}  \cap r( \mo ^{\ir} ( E_K))$ is finite and for any of $x \in X_{\dot q}^{\ir} $, $r^{-1}(x) $ consists of a single regular irreducible representation $\rho$ which satisfies for any generator $\xi\in H^1_\rho(E_K)$:
$$\langle r^*\xi, \mu' \rangle \ne 0$$
\end{enumerate}
If these assumptions are satisfied, any representation $\rho \in X_{\dot q } ^{\ir} \cup X_{\dot q} ^{\ab} $ has a unique extension $\rho_K \in \mo (E_K)$. We denote by $\T_{\mu'} (\rho_K)$ the torsion of $\rho_K$ normalized by $\mu'$. It is the complex number number defined by 
$$ \T_{\mu'} (\rho_K) =   \T ( \rho_K) / \om ( \mu', \cdot )  .
$$  

For any real $\dot q$, consider the flat section of $L^k \rightarrow E$ over the affine line $\dot q \la ' + \R \mu'$ which is equal to $1$ at $\dot q \la'$. It is explicitly given by
$$ \dot q \la ' + x \mu' \rightarrow \exp ( 2 i k \pi \dot q x) .$$
When $\dot q \in ( 2k)^{-1} \Z$, this section descends to a section $t^k_{\dot q}$ of the $k$-th power of the Chern-Simons bundle $L_{CS} \rightarrow \mo ( \Si)$ over the circle $C_{\dot q}$. 

\begin{theo}\label{conj-witten-gen}
 Let $0< q_m < q_M < 1/2$ such that any $\dot q \in [ q_m , q_M]$ satisfies the previous assumptions I1 and I2. Then for any integers $k>0$ and $\ell$ such that $\ell/2 k \in  [ q_m , q_M]$ we have
$$ Z_{k, \ell }  =  k^{- 1/2} \sum_{\rho \in X^{\ab}_{\dot q } \cup X^{\irr}_{\dot q } } e^{i \frac{m ( \rho) \pi }{4}} k^{n ( \rho)} \la_k ( \rho) \langle \CS^k ( \rho_K), t^k_{\dot q}  ( \rho)  \rangle   + O( k^{-\infty}) 
  $$  
where the $O(k^{-\infty})$ is uniform with respect to $k$ and $\ell$ and for any $\rho$,  $m( \rho)$ is an integer, $n( \rho)  = 0$, $-1/2$ according to whether $\rho$ is in  $X^{\irr}_{\dot q} $ or $X^{\ab}_{\dot q} $. Furthermore, $( \la_k ( \rho))$ is a sequence of complex numbers admitting an asymptotic expansion of the form 
$$ \la_k ( \rho) =   a_0 ( \rho) + a_1 ( \rho ) k^{-1} + a_2 ( \rho) k^{-2} + \ldots $$
with coefficients  $a_\ell ( \rho )  \in \C$, the leading one being given by
$$ a_0 ( \rho) = \begin{cases} 2^{-3/4} \bigl( \T _{\mu'} ( \rho_K) \bigr)^{1/2}  \text{ if $\rho \in X ^{\ir} _{\dot q}  $} \\ 
 2^{-1/4} \bigl( \T_{\mu'} ( \rho_K) \bigr)^{1/2} \text{ if $\rho \in X^{\ab}_{\dot q}$ } 
\end{cases} $$ 
where $\rho_K \in \mo (E_K)$ is the unique extension of $\rho$, $\T _{\mu'} ( \rho_K)$ is the $\mu'$ normalized torsion of $\rho_K$ and $\CS( \rho_K) \in L_{\CS, \rho}$ is the Chern-Simons invariant of $\rho_K$. 
\end{theo}

\begin{proof} 
Let $e_\ell = Z_k (N, L, \ell, 0) \in V_k ( \Si)$ so that 
$$ Z_{k,\ell} = (\tau_k )^m \langle Z_k ( E_K) , e_{\ell} \rangle $$ 
for some integer $m$. Here $\tau_k = e^{3i \pi /4 - 3 i \pi /2k} $ is the anomaly factor. Let $\Om_{\mu'} \in \delta$ be such that $\Om_{\mu'} ^2( \mu') = 1$. By Theorem 2.2 of \cite{LJ1}, there exists a unique orthonormal basis $(\Psi_\ell) _{\ell \in \Z / 2 k \Z}$ of $\Hilb_k$ such that
$$ T^*_{ \mu'/2k} \Psi_{\ell} = e^{ i \ell \frac{\pi}{ k}} \Psi_\ell, \qquad T^*_{ \la '/2k} \Psi_{\ell} =  \Psi_{\ell+1} 
$$ 
and 
$$ \Psi_\ell (0) = \Theta_k ( 0, \tau) \Bigl( \frac{k}{2 \pi} \Bigr) ^{1/2} \Om_{\mu'}  .$$
By Theorem 2.4 of \cite{LJ1}, there exist integers $n$ and $n'$ such that
$$ e_{\ell} = \frac{e^{ i \pi ( \frac{n}{4} + \frac{ n'}{ 2k} )} }{ \sqrt 2} ( \Psi_{\ell} - \Psi_{- \ell} )  
$$ 
Finally by Proposition 3.2 of \cite{LJ1}, we have the following estimate
$$ \Bigl| \Psi_{\ell} (x) - \Bigl( \frac{ k } { 2 \pi } \Bigr)^{1/4} T^*_{\ell \la'/2k} t^k (x) \otimes \Om_{\mu'} \Bigr| \leqslant C(\delta) e^{-k / C(\delta)}
$$ 
for any $ \delta \in (0,1)$ and $x \in - \frac{\ell}{2k} \la' + [ \delta, \delta] \la' + \R \mu'$. Here $t$ is the holomorphic section of $L \rightarrow E$ whose restriction to $\R \mu'$ is equal to $1$. The final result follows by an application of the pairing formula of Proposition 6.1 of \cite{LJ1} as for Theorem \ref{theo:witt-conj}. 
\end{proof}

\section{The figure eight knot state}\label{sec:huit}

let $E_8$ be the exterior of the figure eight knot. The moduli space $\mo ^{\irr} ( E_8)$ is diffeomorphic to a circle, see for instance \cite{klassen}. The restriction map $\mo ^{\irr} ( E_8) \rightarrow \mo^s( \Si)$ is an immersion with one double point $x_o=\pi(\frac\la 4)$, its image is shown in Figure \ref{fig:carac-huit}.

$$ r( \mo ^{\ir} ( E_8) )  = \{ \pi ( p \mu + q \la ), \;   \cos(2\pi p)-\cos(8\pi q)+\cos(4\pi q)+1 = 0 \} \setminus\{\pi(\mu/2)\}$$
Its intersection with $r( \mo^{\ab}(E_8))$ is reduced to the point $x_o$. So $r( \mo ^{\ir} ( E_8) ) \setminus \{ x_o \} $ is the disjoint 
union of two embedded open arcs in $\mo^{s} ( \Si) \setminus r ( \mo^{\ab} (E_8))$. 

\begin{figure}[htbp]
\centering
  \def\svgwidth{10cm}
 \executeiffilenewer{huit-carac.svg}{huit-carac.pdf}%
 {inkscape -z -D --file=huit-carac.svg %
 --export-pdf=huit-carac.pdf --export-latex}%
 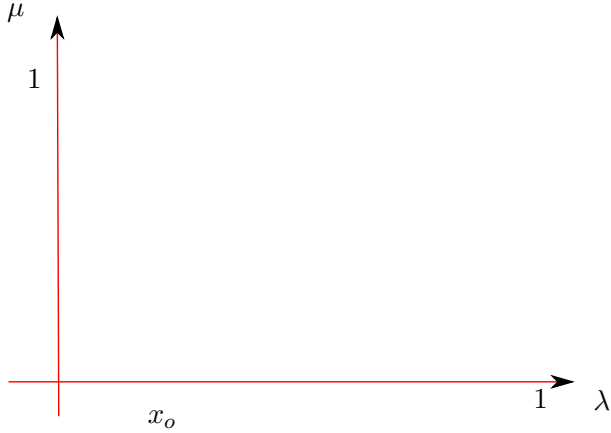%

  \caption{Projection of the character variety of the Figure eight knot}
  \label{fig:carac-huit} 
\end{figure}

As it will be proved in Lemma \ref{sec:racine-torsion},  the bundle $r^* \delta \rightarrow \mo ^{\irr} ( \Si)$ has a section $\si$ such that $\si^4 = \T^2$, with $\T$ the Reidemeister torsion normalized as in Section  \ref{sec:reidemeister-torsion}. The aim of this section is to prove the following theorem. 

\begin{theo}\label{theo:noeud_de_8}
We have on the set $V=\mo^{s} ( \Si) \setminus r ( \mo^{\ab} (E_8)) $
$$Z_k (E_8)   =   \la_k e ^{i \frac{m\pi}{4}} \frac{k^{3/4}}{4\pi^{3/4}} F^k g(\cdot ,k ) + O(k^{-\infty})$$
where the $O( k^{-\infty})$ is uniform on any compact set of $V$, $m$ is an integer
\begin{itemize} 
\item[-] $F$ is a section of $L \rightarrow V$ such that  $F(r(\rho)) = \CS (\rho)$ for any $\rho \in  \mo^{\irr} ( E_8) \setminus r^{-1} (x_o)$ and which satisfies the Cauchy-Riemann equation up to a term vanishing to infinite order along $r( \mo ^{\irr}(E_8) )$. 
\item[-] $g(\cdot , k)$ is a sequence of $\Ci (V,\delta)$ admitting an
  asymptotic expansion of the form $g_0 + k^{-1} g_1 +  \ldots$ with $g_0 (r ( \rho) )= \si( \rho)  $ for any $\rho \in  \mo^{\irr} ( E_8) \setminus r^{-1} (x_o)$.
\item[-] $\la_k$ is a sequence of complex numbers equal to $1 + O(k^{-1})$.    
\end{itemize}
\end{theo}

The result is weaker than Conjecture \ref{conj:irreducible} because of the sequence $\la_k$.  Our proof relies on a $q$-difference equation satisfied by the Jones polynomials of the figure eight knot. By Proposition 4.4 in \cite{LJ1}, the figure eight knot state $Z_k (E_8)$ satisfies the following equation:
\begin{gather} \label{eq:eq_q_difference}
 Q_k Z_k (E_8) = R_k Z^0_k
\end{gather}
Here $Q_k$ and $R_k$ are the operators
\begin{xalignat}{2} \notag
Q_k &=(q^{-1}M^2-qM^{-2})L+(qM^2-q^{-1}M^{-2})L^{-1} +(M^2-M^{-2})(-M^4\\ & -M^{-4}  +M^2+M^{-2}+q^2+q^{-2}),  \label{eq:def_Q} \\  \notag
R_k  & =  ( M^5 + M^{-5}  + M^3 + M^{-3} - ( q^2 + q^{-2} ) ( M + M^{-1} ) ) 
\end{xalignat}
with $q = e^{i \pi /k}$ and  $Z^0_k$ is the section of $\Hilb_k$ given by 
$$ Z^0_k = \frac{1}{2i \sqrt k} 
\sum_{\ell \in \Z / 2 k \Z } \Psi_\ell.   
$$

\subsection{Toeplitz operator and Lagrangian states}  

In this section we introduce the analytical tools necessary to prove Conjecture \ref{conj:irreducible} for the figure eight knot. It is convenient to work on the torus $T =  E/R$, because the quotient $E/ R \rtimes \Z_2$ is singular and $E$ is not compact. Abusing notation, we denote again by $\delta$ and $L$ the line bundles over $T$ obtained from the line bundles $\delta$ and $L$ over $E$. So $\Hilb_k$ consists in the holomorphic sections of $L^k \otimes \delta \rightarrow T$. The group $(R\rtimes \Z_2)/R=\Z_2$ acts linearly on $\Hilb_k$, the invariant subspace being $\Hilb_k^{\alt}$. 

\subsubsection{Toeplitz operators}

Let $\Hilb_k^2$ be the space of sections of $L^k \otimes \delta \rightarrow T$ which are locally of class $L^2$. We denote by $\Pi_k$ the orthogonal projector of $\Hilb^2_k$ onto the (closed) finite dimensional subspace $\Hilb_k$.  For any bounded function $f \in \Ci (T)$ we denote by $M(f)$ the operator acting on $\Hilb^2_k$ by multiplication by $f$.

A {\em Toeplitz operator} is a family $(T_k \in \End ( \Hilb_k), \; k \in \Z_{>0})$ of the form 
$$ T_k  =  \Pi_k M(f(\cdot, k ) ) + R_k : \Hilb_k \rightarrow \Hilb_k , \qquad k=1,2, \ldots $$
where $(f( \cdot, k))_k$ is a sequence of $\Ci ( T )$ which admits an asymptotic expansion $$ f( \cdot, k ) = f_0 + k^{-1} f_1 + \ldots $$ for the $\Ci$ topology, with coefficients $f_0, f_1, \ldots \in \Ci ( T)$. Furthermore the family $(R_k \in \End ( \Hilb_k), \; k \in \Z_{>0}) $ is a $O(k^{-\infty})$, i.e. for any $N$, $ \| R_k \| =O(k^{-N})$. Here $\|\cdot \|$ is the uniform norm of operators. 

As a result, the coefficients $f_0, f_1, \ldots$ are uniquely determined by the family $(T_k)_k$. We call $f_0$ the principal symbol of $(T_k)$ and $f_1 -\frac{1}{2} \Delta f_0$ the subprincipal symbol of $(T_k)$.

It follows from Theorem 3.1 of \cite{LJ1} that the families $(Q_k)$ and $(R_k)$ defined in (\ref{eq:def_Q}) are Toeplitz operator of $\Hilb_k$. The  principal and subprincipal symbols of $(Q_k)$ are given by: 
\begin{xalignat*}{1}
 & f_0 (x)  = -4i\sin(4\pi q)\big(\cos(2\pi p)-\cos(8\pi q)+\cos(4\pi q)+1) \\
 & f_1 (x) = -8\pi  \cos(4\pi q)\sin(2\pi p)
\end{xalignat*}
where $x = [ p \mu + q \la ] \in T$. 
\subsubsection{Microlocal solution} 

Let $U$ be an open set of $E/R$. Consider a family $( \Psi_k \in \Ci( U, L^k \otimes \delta), k \in \Z_{>0})$.  We call such a family a {\em local state} on $U$. For any $N \in \Z$, we say that $(\Psi_k)$ is a $O(k^{N})$ if for any compact subset $K$ of $U$,  there exists $C$ such that $$| \Psi_k ( x) | \leqslant C k^{N}, \qquad \forall x \in K.$$ We say that $(\Psi_k)$ is admissible (resp. a $O(k^{-\infty})$) if for some integer $N$ (resp. for any negative integer $N$), it is a $O(k^{N})$.

Let $(T_k)$ be a Toeplitz operator and assume that $(\Psi_k)$ is an admissible local state on $U$. We say that $(\Psi_k)$ is a microlocal solution on $U$ of 
\begin{gather}\label{eq:eq}
T_k \Psi_k =0 
\end{gather}
if for any $x$ in $U$, there exists a function $\varphi \in \Ci(M)$ such that $\operatorname{supp} \varphi \subset U$, $\varphi \equiv 1$ on a neighborhood of $x$ and 
$$ \Pi_k (\varphi \Psi_k) = \Psi_k + O(k^{-\infty}), \quad T_k ( \Pi_k ( \varphi \Psi_k)) = O( k^{-\infty})$$
on a neighborhood of $x$.

As expected although the proof is not so obvious, for any family $( \Psi_k \in \Hilb_k, k \in \Z_{>0})$ such that $T_k \Psi_k =0$, the restriction $(\Psi_k |_U )$ is a microlocal solution on $U$.  Observe that the set ${\mathcal{S}}$ of microlocal solution of (\ref{eq:eq}) is a ${\mathcal{R}}$-module where ${\mathcal{R}}$ consists in the admissible sequences $(\la_k)$ of complex number. Here admissible means that $\la_k = O( k^{N}) $  for some $N$. Furthermore ${\mathcal{S}}$ contains as a submodule the set of local state $(\Psi_k)$  which are a $O(k^{-\infty})$. It is known that if the principal symbol of $(T_k)$ does not vanish on $U$, ${\mathcal{S}}$ consists only on the $O(k^{-\infty})$ local states. This assertion is equivalent to the fact that the microsupport of any admissible family $( \Psi_k \in \Hilb_k)_k$ satisfying  $T_k \Psi_k = 0 $ is contained into the zero set of the principal symbol of $(T_k)$. 

Consider now the figure eight knot state $Z_k^8$. Let $X$ be the following subset of $T$ 
\begin{gather} \label{eq:def_X}
 X = \{   \cos(2\pi p)-\cos(8\pi q)+\cos(4\pi q)+1 = 0  \} .
\end{gather}
The set $X$ intersects the circle $\{ p = 0 \}$ at two points $P$ and $P'$.  Our aim is to describe the asymptotic behaviour of $(Z_k (E^8))$ on a neighborhood of $X \setminus \{ P, P'\}$. 

By Theorem 5.4 of \cite{LJ1}, the microsupport of the left hand side $R_k Z_k^0$ of Equation (\ref{eq:eq_q_difference}) is contained in $\{ p = 0 \}$. So the restriction of the knot state $Z_k(E^8) $ to the open set $T \setminus \{ p = 0 \}$ 
is a microlocal solution of $Q_k \Psi_k =0$.

\subsubsection{Lagrangian microlocal solution}

Consider a Toeplitz operator $(T_k)$ of $\Hilb_k$ with a real valued principal symbol $f_0$. Then on a neighborhood of any regular point of $f_0$, the microlocal solutions of $T_k \Psi_k = 0$ are Lagrangian states supported by $f^{-1}_0 (0)$ in the following sense.

Let $U$ be any open set of $E$ such that $I= f_0^{-1}(0) \cap U$ is diffeomorphic to an interval and $d f_0$ does not vanish on $I$. Let $t$ be a flat section of $L \rightarrow I$ with constant norm equal to 1. Let $\si \in \Om^1 ( I)$ be a non vanishing solution of the following transport equation
\begin{equation} \label{eq:transport}
\sL_{X_0}\si +2 if_1 \si=0
\end{equation}
Here $X_0$ is the restriction to $I$ of the Hamiltonian vector field of $f_0$.  

\begin{theo} \label{theo:lagr-micr-solut}
The equation $T_k \Psi_k =0$ has a microlocal solution $U$ of the form $F^k g (\cdot ,k)$ 
where 
\begin{itemize} 
\item[-] $F$ is a section of $L \rightarrow U$ satisfying the Cauchy-Riemann equation up to a section vanishing to infinite order along $I$ and $F|_I = t$. 
\item[-]   $g (\cdot ,k) $ is a sequence $\Ci (U,\delta)$ admitting an asymptotic expansion for the $\Ci$ topology of the form $g_0 + k^{-1} g_1 + \ldots$
\item[-] The leading coefficient $g_0$ satisfies  $j^* \varphi ( g_0^{\otimes 2}) = \si $, where $\varphi$  is the squaring map of the half-form bundle $\delta$ and $j$ is the injection $I \rightarrow M$.  
\end{itemize} 
Furthermore for any admissible local  solution  $(\Psi_k \in \Ci ( U, L^k \otimes \delta), \; k \in \Z_{>0}) $ of $T_k \Psi_k =0$ on $U$, we have 
$$ \Psi_k = \la_k F^k g (\cdot ,k) + O(k^{-\infty})$$ 
for some admissible sequence $(\la_k)$ of complex number.
\end{theo}

As in the previous section, let ${\mathcal{S}}$ be the ${\mathcal{R}}$-module of microlocal solutions of $T_k \Psi_k = 0 $ on $U$. The previous theorem shows that ${\mathcal{S}}/ O(k^{-\infty})$ has dimension 1 and gives the asymptotic behaviour of a generator.  

For the figure eight knot state, we can apply this result on each of the four open rectangles: 
\begin{gather} \label{eq:rectangle_Ui}
U_i = \bigl\{ [p\mu+q\la], \;p \in (0,1),\; q \in (\tfrac{i}{4},\tfrac{i+1}{4}) \bigr\}, \qquad i=0,1,2,3
\end{gather}
Indeed, since each $U_i$ is contained in $T \setminus \{ p = 0 \}$, the restriction of $Z_k (E^8)$ to $U_i$ is a microlocal solution of $Q_k \Psi_k = 0 $. Furthermore, $I_i = U_i \cap f^{-1}_0 (0)$ is connected.

Observe also that $U_1 \cup U_2 \cup U_3 \cup U_4$ is a neighborhood of $X \setminus \{ P, P'\}$.  So it is sufficient to consider these open sets to prove Theorem \ref{theo:noeud_de_8}. Furthermore, $\Z_2$ acts on $T$ by switching $U_1$ with $U_4$  and $U_2 $ with $U_3$. Since $Z_k(E^8)$ is $\Z_2$-invariant, it is actually sufficient to consider two of these rectangles. We will use in Section \ref{sec:symmetry} the fact that $Z_k(E^8)$ satisfies an additional symmetry which reduces everything to one rectangle. 

Another remark is that we can define the flat section $t$ from the Chern-Simons invariant in the following way. The map sending $x \in I_i$ to the representation $\rho \in \mo ^{\irr} ( E_8)$ such that $[x]= r( \rho) $ is a diffeomorphism from $I_i$ onto an open set of $  \mo ^{\irr} ( E_8)$. Let $t$ be the section of $L \rightarrow I_i$ given by $t( x) = \CS ( \rho)$ where $[x]= r( \rho) $. This section is flat as recalled in Section \ref{sec:chern-simons-invar}.

\subsection{Torsion and transport equation} 

We prove in this part that the torsion of the exterior of the figure eight knot satisfies the transport equation \eqref{eq:transport} where $f_0$ and $f_1$ are the principal and subprincipal symbols of the operator $ (Q_k)$. 

More precisely since $f_0$ and $f_1$ do not descend to the quotient $\mo ( \Si)$, we lift everything to the torus $T$. The set $X$ defined in (\ref{eq:def_X}) is the reunion of two immersed circles and two points. Let $\si$ be the restriction of $ \mathrm{d} p / (1 - 4 \cos ( 4 \pi q ) )$ to these circles. Since $f_0$ vanishes on $X$, the Hamiltonian vector field of $f_0$ is tangent to $X$.  

\begin{prop} 
One has that $\sL_{X_0}\si +2 if_1 \si=0$. 
\end{prop} 

Since $\si$ does not vanish anywhere, its absolute value $|\si|$ satisfies the same transport equation. This density is up to a constant the torsion $\T$ defined in Subsection \ref{sec:examples}.
\begin{proof} 
Introduce the following functions of $T$
\begin{gather*} 
h = \cos ( 2 \pi p ) + 1 - \cos ( 8 \pi q ) + \cos ( 4 \pi q ) \\
a = 1 - 4 \cos ( 4 \pi q ), \qquad b =\sin ( 4 \pi q ). 
\end{gather*}
So $X$ is the zero set of $h$, $f_0 = -4i b h $  and $\si = \mathrm{d} p / a $. We have 
\begin{xalignat}{2}  \notag
\mathrm{d} h = & - 2 \pi \sin ( 2\pi p ) \mathrm{d} p + \bigl(  8 \pi \sin ( 8 \pi q ) - 4 \pi  \sin ( 4 \pi q) \bigr) \mathrm{d} q \\ \notag
 = & - 2 \pi \sin ( 2\pi p ) \mathrm{d} p + \bigl(  16 \pi \sin ( 4 \pi q ) \cos ( 4 \pi q)  - 4 \pi  \sin ( 4 \pi q) \bigr) \mathrm{d} q \\ \label{eq:dh}
= & - 2 \pi \sin ( 2\pi p ) \mathrm{d} p - 4 \pi b a \mathrm{d} q 
\end{xalignat}
Let $Y$ be the Hamiltonian vector field of $h$. Since $h$ vanishes on $X$, the Hamiltonian vector field of $f_0$ coincides on $X$ with $-4i b Y$. Simplifying the factor $-4 i $,  we have to prove that 
\begin{gather}  \label{eq:tobeproved}
\sL_{b Y}\si + 4 \pi \cos (4 \pi q ) \sin ( 2 \pi p )   \si=0
\end{gather} 
We have that 
\begin{xalignat}{2} \notag
 \sL_{b Y}\si = & b \sL_Y \si + \si (Y) \mathrm{d} b \\ \label{eq:dersi}
= & b \Bigl( \frac{  \mathrm{d} \sL_Y p }{a}  - \frac{ \sL_Y a }{a^2}  \mathrm{d} p \Bigr) + \frac{ \sL_Y p } {a}  \mathrm{d} b  
\end{xalignat}
Since the symplectic form of $T$ is $ 4 \pi dp \wedge d q $, we have by (\ref{eq:dh})
$$ \sL_Y p = - \frac{1}{4 \pi} \partial_q h = ba $$
Furthermore $\sL_Y a = 16 \pi \sin ( 4 \pi q)  \sL_Y q = 16 \pi  b\sL_Y q  $ and 
$$  \sL_Y q =  \frac{1}{4 \pi} \partial_p h = - \frac{1}{2} \sin ( 2 \pi p)$$
where we used (\ref{eq:dh}) again. Inserting these expressions into (\ref{eq:dersi}) we obtain
\begin{xalignat*}{2} 
 \sL_{b Y}\si = & b  \Bigl( \frac{  \mathrm{d}( ba ) }{a}  + 8 \pi \frac{b  }{a^2} \sin ( 2 \pi p)   \mathrm{d} p \Bigr) + b   \mathrm{d} b   \\
= &  2 b \mathrm{d} b + \frac{ b^2}{a} \mathrm{d} a + 8 \pi \frac{ b^2 }{a^2} \sin ( 2 \pi p) \mathrm{d} p \\
= &   2 b \mathrm{d} b + \frac{ b^2}{a} \mathrm{d} a - 16 \pi \frac{ b^3 }{a} \mathrm{d} q 
\end{xalignat*}
where we used that $\mathrm{d} h =0$ on $X$ and (\ref{eq:dh}). Finally since  $\mathrm{d} a =  16 \pi b \mathrm{d} q$, we obtain 
\begin{xalignat*}{2}
 \sL_{b Y}\si = &  2b \mathrm{d} b \\
= & 8 \pi b \cos ( 4 \pi q ) \mathrm{d} q  \\
= & - 4 \pi \cos ( 4 \pi q) \sin ( 2 \pi p) \frac{\mathrm{d} p}{a} 
\end{xalignat*}
where we used again  that $\mathrm{d} h =0$ on $X$ and (\ref{eq:dh}). Since $\si = \mathrm{d} p /a$, this shows (\ref{eq:tobeproved}) and concludes the proof. 
\end{proof}

\subsection{Symmetry} \label{sec:symmetry}

Let $\Gamma = R \rtimes \Z_2$ and $\Ga'= R' \rtimes \Z_2$ with $R'$ the lattice $\mu \Z \oplus \frac{\la }{2} \Z$. It is easily checked that we can extend the action of $\Ga$ on the bundles $L$ and $\delta$ to $\Ga'$ in such a way that 
$$ (\la /2, 1 ). (x, v)  = (  x + \la/2, - v) , \qquad  (x, v) \in \delta 
$$ 
$$ (\la /2  , 1  ) .( x, v)  =  (  x +  \la/2 , e^{\frac{i}{2} \om (\la/2, x )} v) , \qquad  (x, v) \in L
$$  
The second formula is just the action of the element $(\la/2, 1)$ of the Heisenberg group. 
By Proposition \ref{sec:symm-knot-state}, the state of any knot is a $\Ga'$-invariant holomorphic section of $L^k \otimes \delta \rightarrow E$.

\subsubsection{On the half-form bundle} 

As previously denote by  $\tau$ the action of the generator of $\Ga'/\Ga$ on $\mo ( \Si) \simeq E/ \Ga$ and by $\si$ the action of $\rho_{-1}$ on $\mo (E_K)$. The restriction map $r: \mo (E_K) \rightarrow \mo ( \Si)$ intertwines $\tau $ and $\si$. Denote by $\tau_\delta$ the action of the generator of $\Ga'/\Ga$ on the half-form bundle $\delta \rightarrow \mo ( \Si)$. 

Consider now the complement of the figure eight knot. Its set of irreducible representations is a circle $C$, preserved by $\si$.  Denote by $\delta_C \rightarrow C$  the pull-back of $\delta \rightarrow \mo ( \Si)$ by $r$ and $\sigma_\delta$ the pull-back of $\tau_\delta$. We have an isomorphism $\varphi_C$ between $\delta_C^4$ and $(T^*C \otimes \C)^2$ defined as follows
$$ u^2 \in \delta_{C,p}^4  = \delta^4 _{r(p)} \simeq \bigl( \wedge^{1,0} T_{r(p)}^* \mo^s ( \Si) \bigr)^2 \rightarrow (r^*u)^2 \in \bigl( T^*C \otimes \C \bigr)^2 $$
This isomorphism intertwines the morphisms $\si_\delta^4 $ and $(\sigma^*)^2$.

The square of the torsion $\T$ is a section of $\bigl( T^*C \otimes \C \bigr)^2 $. 
\begin{lem} \label{sec:racine-torsion}
The bundle $\delta_C\rightarrow C $ has a smooth section $g$ such that $\varphi_C ( g^4) = \T^2$. It is unique up to multiplication by a power of $i$ and is $\si_\delta$-invariant. 
\end{lem}
\begin{proof} 
Let us prove the existence of a non-vanishing section $g$ satisfying $g^4 = \T^2$. 
We identify smoothly $C$ with $\R / 2 \Z   $ in such a way that the involution $ \si $ is the map sending $[p] \in C$ to $[-p] \in C$. The two fixed points $[0]$ and $[1]$ are sent by $r$ to $\pi ( \la / 4)$. Let $C_+ = \{ [p]; p \in ]0,1[ \}$ and $C_-  = \{ [p]; p \in ]1,2[ \}$. 

Define $g$ over $C_+$ in such a way that $g^4 = \T^2$. Since $\si^* \T^2 = \T^2$, we can extend $g$ over $C_+ \cup  C_-$ in such a way that $g$ is $\si_\delta$ -invariant and $g^4 = \T^2$ is still satisfied. Because of equation (\ref{eq:torsion_8}), $g$ has left and right limits at $[0]$ and $[1]$. By symmetry reason, the left and right limits are the same, so that $g$ extends continuously to $C$. 

\end{proof}

\subsubsection{The operator $Q$}

Let $I$ be the endomorphism of $\Hilb_k$ given by $I \Psi(x) = \Psi( -x)$ where we view the elements of $\Hilb_k$ as $R$-invariant section over $E$. So $\Hilb_k^{\alt} = \ker ( I + \id_{\Hilb_k}) $.  

\begin{lem} \label{lem:symmetrie_operateur}
The operator $Q$ defined in (\ref{eq:def_Q}) commutes with $T^*_{\la/2}$ and anticommutes with $I$. 
\end{lem}

\begin{proof} 
Since $M$ and $L$ are respectively the pull-back by the actions of the elements $(\mu/2k , 1)$ and $( -\la / 2k ,1)$ of the Heisenberg group, we have that
$$  T^*_{\la/2} M = - M T^*_{\la/2}, \qquad T^*_{\la/2} L = L T^*_{\la/2}$$ 
Since  $Q$ is a polynomial expression in the variables $M^2$ and $L$, it commutes with $T^*_{\la/2}$. That $Q$ anticommutes with $I$ follows from the relations $M I = I M^{-1}$ and $L I = I L^{-1}$.   
\end{proof}

\subsubsection{Symmetric microlocal solution of $Q\Psi_k = 0 $}

Let us now work on the torus $T = E/R$. The action of $\Ga'$ on $E$ descends to an action of $\Ga'/R \simeq \Z_2 \times \Z_2$ on $T$. The two generators act on $T$ by $[x] \rightarrow [-x]$ and $[x] \rightarrow [x + \la/2 ]$. 

Observe that $\Ga'/R$ acts simply transitively on the family $(U_1, U_2, U_3, U_4)$. Let  $V = U_1 \cup U_2 \cup U_3 \cup U_4$.  We deduce from Theorem \ref{theo:lagr-micr-solut} and Lemma \ref{lem:symmetrie_operateur} that the module of $\Ga'/R$-invariant microlocal solutions of $Q \Psi_k = 0 $ on $U$ modulo $O(k^{-\infty})$ is a one-dimensional ${\mathcal{R}}$-module. Furthermore this module has a generator $(\Psi_k \in \Ci(V, L^k \otimes \delta),\; k \in \Z_{>0})$ of the form 
\begin{gather} \label{eq:Psi_k}
 \Psi_k =  \frac{k^{3/4}}{4\pi^{3/4}} F^k \tilde g ( \cdot, k)  
\end{gather}
where $F$ and $g$ satisfies the same assumption as in Theorem \ref{theo:lagr-micr-solut}. 
By Lemma \ref{sec:racine-torsion} and Lemma \ref{sec:symmetry-CS}, we can also  assume that for any irreducible representation $\rho \in \mo^{\ir} ( E_8)$ and $x \in E$ such that $\pi (x) = \rho$, we have that
$$ F( [x]) = \CS(\rho), \qquad  \tilde g ( x, k ) = g ( \rho) +O(k^{-1}) 
$$  
where $g$ satisfies the same assumption as in Lemma \ref{sec:racine-torsion}. Finally, we have that 
\begin{gather}\label{eq:semi_classic_etat}
 Z_k ( E_8) = \la_k \Psi_k + O(k^{-\infty})  \quad \text{ on } V
\end{gather} 
where $(\la_k)$ is an admissible sequence of complex numbers.

\subsection{End of the proof}

In this part, we show that the sequence $(\la_k)$ in (\ref{eq:semi_classic_etat}) satisfies
\begin{gather} \label{eq:estim_lak}
 \la_k = e^{im \frac{\pi}{4}} + O( k^{-1})
\end{gather}
for some integer $m$.  To prove this we use the fact that the surgery of index 1 on the figure eight knot is homeomorphic to the Brieskorn sphere $\Sigma ( 2,3,7)$. Hikami proved in \cite{hikami2} that the Brieskorn spheres satisfy the Witten asymptotic conjecture. In the following theorem we present the leading order term for $\Si ( 2, 3, 7)$.  

\begin{theo}\cite{hikami2} \label{theo:Brieskorn}
The WRT invariant of the Brieskorn sphere $\Sigma(2,3,7)$ satisfies
$$ Z_k ( \Sigma(2,3,7)) = \frac{e^{i m \frac{\pi}{4}}}{2}   \sum_{i=1,2} \CS (\rho_i)^k  \T ( \rho_i) ^{1/2}   + O(k^{-1}) $$
for some integer $m$ with $\mo ^{\ir} ( \Sigma(2,3,7)) = \{ \rho_1, \rho_2 \}$. Furthermore $$ \T ( \rho_1) = \frac{2 ^{3/2}}{7^{1/2}} \sin \bigl( \frac{2\pi}{7} \bigr), \quad \CS( \rho_1) = e^{-i \frac{25\pi}{84}} $$
and 
$$ \T ( \rho_2) = \frac{2 ^{3/2}}{7^{1/2}} \sin \bigl( \frac{3\pi}{7} \bigr) , \quad \CS( \rho_2) = e^{i \frac{47 \pi}{84}}  .$$  
\end{theo}  

Let $N$ be a solid torus with boundary $\Si$ such that $\mu + \la$ vanishes in $H_1 (N)$. Then$$ \langle Z_k (E_8) , Z_k (N) \rangle = Z_k ( \Sigma(2,3,7)) $$
Recall that the microsupport of $Z_k (N)$ is the circle  $ \Ga = \{ [s ( \la + \mu ) ]/ s\in \R \}$ whereas the microsupport of $Z_k(E^8)$ is contained in $\{ p =0 \} \cup X$ where $X$ is the characteristic set (\ref{eq:def_X}). Hence, for any closed neighborhood $K$ of $\Ga \cap ( \{ p = 0 \} \cup X)$ we have 
$$ \langle Z_k (E_8) , Z_k (N) \rangle = \int_{K} \bigl( Z_k (E_8)   , Z_k (N) \bigr)_{L^k \otimes \delta}  (x) \;  \mu_T ( x)  +O(k^{-\infty}) $$
where $( \cdot, \cdot  )_{L^k \otimes \delta}$ denote the pointwise hermitian product of $L^k \otimes \delta$ at $x \in T$ and $\mu_T$ is the Liouville measure of $T$. 
 
Let  $K_1$ and $K_2$ two disjoint compact sets neighborhood of $ \Ga \cap \{ p = 0\} $ and $\Ga \cap X$ respectively. Let  $I_1$ and $I_2$ the integrals of $( Z_k (E_8) , Z_k (N))_{L^k \otimes \delta}   \mu_T $ over $K_1$ and $K_2$ respectively. Using that $Z_k (N)$ and $Z_k (E^8)$ are Lagrangian states on a neighborhood of $\{ 0_T \} = \Ga \cap \{ p = 0 \}$, we can estimate $I_1$ as in the proof of Theorem \ref{theo:witt-conj}. We deduce that
$$  I_1   = O(k^{-3/2} ) $$ 
Assume that $K_2$ is contained in the reunion $U$ of the open sets $U_i$ defined in (\ref{eq:rectangle_Ui}). Because of (\ref{eq:semi_classic_etat}), we have
$$ I_2 =    \la_k \int_{K_2}  \bigl( \Psi_k   , Z_k (N) \bigr)_{L^k \otimes \delta}  \; \mu_T + O(k^{-\infty})$$ 

\begin{lem} We have that 
$$ \int_{K_2} \bigl( \Psi_k   , Z_k (N) \bigr)_{L^k \otimes \delta}  \;  \mu_T   = e^{i m' \frac{\pi}{4}}  Z_k ( \Sigma(2,3,7)) + O(k^{-1} ) $$
for some integer $m'$.  
\end{lem}

\begin{proof} Since $\Ga$ intersects transversally $X$, we can argue exactly as in the proof of Theorem \ref{theo:witt-conj}. We deduce that 
$$ \int_{K_2} \bigl( \Psi_k   , Z_k (N) \bigr)_{L^k \otimes \delta}  \;  \mu_T = 2^{-1}  \sum_{i=1,2} e^{i m_i \frac{\pi}{4}}  \CS (\rho_i)^k  \T ( \rho_i) ^{1/2}   + O(k^{-1}) $$
where $\rho_1$ and $\rho_2$ are the two irreducible representations of $\Si ( 2, 3, 7)$ as in Theorem \ref{theo:Brieskorn}. To conclude, we have to show that the two integers $m_1$ and $m_2$ are equal. Let us recall the half-form pairing. Given two transversal lines $F_1$ and $F_2$ of $E$, we define a sesquilinear map 
$$ \delta \times  \delta \rightarrow \C, \qquad \al_1 , \al_2 \rightarrow \langle \al_1, \al_2 \rangle_{F_1, F_2}$$
Here we denote by $\delta$ the complex line and not the bundle, its square is identified with the canonical line of $(E,j)$.  The square of this pairing is given by 
\begin{gather} \label{eq:carre_pair}
 \langle \al_1, \al_2 \rangle_{F_1, F_2}^2 = i \frac{ \al_1^2( X_1)  \al_2 ^2 (X_2) }{ \om (X_1, X_2) } 
\end{gather}
where $X_i $ is any vector non-vanishing vector in $F_i$ for $i = 1,2$. The way the square root is determined is explained in \cite{l1}.  Here we will only use the fact that the pairing depends continuously on $F_1$ and $F_2$. 
As an application of stationary phase lemma, we obtain  
$$ \int_{K_2} \bigl( \Psi_k   , Z_k (N) \bigr)_{L^k \otimes \delta}  \;  \mu_T \sim 2^{-1}  \sum_{i=1,2} \CS (\rho_i)^k  a_i $$
where $a_i$ is the pairing 
\begin{gather}\label{eq:pa}
a_i = \langle g ( x_i ), \sin ( 2 \pi q( x_i) ) \Om _{\la + \mu } \rangle _{ T_{r(x_i) } r ( \mo (E_K)) , (\la + \mu ) \R}
\end{gather}
Here $g$ is a section of $ \delta_C \rightarrow C $ satisfying the assumption of Lemma \ref{sec:racine-torsion}. The representations $x_1,x_2\in\mo ^{\ir} (E_K)$ have restrictions $\rho_1$ and $\rho_2$ respectively. $q(x_i)$ is the second coordinate of $r(x_i)$ in the base $\mu, \la$. Finally $\Om_{\la + \mu } \in \delta$ satisfies $\Om_{\la +\mu}^2 (\la +\mu) =1$.    

Observe that for any $x \in \mo ^{\ir} (E_K)$, the tangent space $ T_{r(x) } r ( \mo (E_K))$ intersects transversally the line directed by $\mu + \la$. Furthermore one can choose the determination of $q(x)$ such that it depends continuously on $x \in \mo ^{\ir}( E_K)$ and satisfies $1/6 < q(x) < 1/3$.  Since the section $g$ is continuous, the right hand side of (\ref{eq:pa}) depends continuously on $x \in \mo ^{\ir} (E_K)$ and does not vanish. Its fourth power being negative by Equation (\ref{eq:carre_pair}), its phase is constant. 
\end{proof}

Collecting together the previous estimates we obtain that 
$$ \bigl( 1 - \la_k  e^{i m ' \frac{ \pi}{4}} \bigr) Z_k ( \Si ( 2, 3, 7)) = O(k^{-1})$$
Using that $\T ( \rho_1) \neq \T ( \rho_2)$, we deduce from Theorem \ref{theo:Brieskorn} that 
$$  \bigl| Z_k (  ( \Si ( 2, 3, 7)) \Bigr| \geqslant \ep 
$$  
for some positive $\ep$ which does not depend on $k$. This implies (\ref{eq:estim_lak}).

\subsection{Regular slopes of the figure eight knot}
By Theorem 4.1, our conjectures on the knot state hold for the figure eight knot. Using Theorem 3.5, we can conclude that the Witten conjecture holds for the Dehn filling with parameters $p,q$ on the figure eight knot as soon as the non-degeneracy conditions H'1 and H'2 are satisfied. We collect in the following proposition the state of our knowledge on that question.
\begin{prop}\label{prop:reg-slope}
Let $p,q$ be two positive coprime integers and consider the hypothesis H'1 and H'2 concerning the Dehn fillings of the figure eight knot with parameters $p,q$. 
\begin{itemize}
\item[-] H'1 is satisfied if and only if $p\ne 0\mod 4$.
\item[-] H'2 is satisfied if $p/q<2\sqrt{5}$.
\item[-] H'2 is satisfied if and only if the polynomial $$\Phi_{p,q}=X^{2p}-X^{p+4q}+X^{p+2q}+2X^p+X^{p-2q}-X^{p-4q}+1$$ has no simple roots on the unit circle except $-1$ ($p=1\mod 2$) and $\pm i$ ($p=0\mod 4$).
\item[-] H'2 is satisfied if $0<2\sqrt{5}q<p\le 200$.
\item[-] Let $\ell$ be an odd prime divisor of $p$ and $\xi$ be a root of $\xi^2-\xi/2+1$ in $\mathbb{F}_{\ell^2}$. If $2^{8q}\ne \xi^{\pm p}$, then $H'2$ holds.
\end{itemize}
\end{prop}
We conjecture that the hypothesis H'2 holds for any $p,q$: we do not know any theoretical reason why such a rigidity result should hold.
\begin{proof}
Set $F(x,y)=\cos(x)+1-\cos(4y)+\cos(2y)$ such that the characteristic variety of the figure eight knot is given by 
$$X_8=\{2\pi x\mu+2\pi y\la \in E, F(x,y)=0\text{ and }(x,y)\notin \mu(1/2+\Z)+\frac{\la}{2}\Z \}$$
The real double points of $X_8$ are located at $2\pi(\mu\Z+\la(1/4+1/2\Z))$ and correspond to 2 representations of the figure eight knot complement in $\su$ whereas the points $2\pi(\mu(1/2+\Z)+\frac{\la}{2}\Z)$ are imaginary double points and correspond to representations in SL$(2,\R)$.

One can parameterize $I_{p/q}$ by setting $x=pt$ and $y=qt$. This line meets the real double points if and only if $p$ is divisible by 4, proving the first point. 
A computation shows that the slopes of the tangents of $X_8$ at the double points are $\pm 2\sqrt{5}$. Moreover, on each branch of $X_8$, the slope varies continuously in $\pm [2\sqrt{5},+\infty]$. This shows that if $p/q<2\sqrt{5}$ the sets $I_{p/q}$ and $X_8$ meet transversally: this shows the second point.

The hypothesis H'2 is equivalent to saying that $I_{p/q}$ has only first order contact with $X_8$. A way to formulate this condition is that the map $t\mapsto F(pt,qt)$ has only simple zeros. Writing $X=\exp(it)$, we deduce immediately the third point. The zero $-1$ correspond to an intersection of $I_{p/q}$ with the imaginary double points (which can be disregarded) whereas the zeroes $\pm i$ correspond to an intersection with the real double points. The intersection of $I_{p/q}$ with each branch meeting at those points is transversal because the slope is irrational. Therefore, these double points can also be disregarded in order to state H'2.

The third item comes from the non vanishing of the discriminants of the following polynomials that we computed with the computer.
$$
\left\{\begin{array}{ll}
\Phi_{p,q}& p=2\mod 4\\
\Phi_{p,q}/(X+1)^2& p=1,3\mod 4\\
\Phi_{p,q}/(X^2+1)^2 & p=0\mod 4
\end{array}\right.
$$
Finally, the last point comes from a computation of the common zeros of $\Phi_{p,q}$ and $\Phi'_{p,q}$ in the field $\overline{\mathbb{F}}_\ell$. We compute in $\mathbb{F}_l$: 
$$\Phi'_{p,q}=2qX^{p-1}(X^{2q}-X^{-2q})(-2X^{2q}+1-2X^{-2q}).$$ 
A root $\alpha$ of $\Phi'_{p,q}$ satisfies either $\alpha=0$, $\alpha^{4q}=1$ or $\alpha^{2q}+\alpha^{-2q}=1/2$. 
We check now whether these roots can be also roots of $\Phi_{p,q}$. The first cannot, the second neither because $p$ and $q$ are coprime (with exceptions depending on $p\mod 4$) and the third gives the condition of the proposition.
We observed that this condition is satisfied very often but not always, the simplest case when it is not being $p=83$ and $q=1$. 
\end{proof}

\end{document}